\documentclass[12pt]{article}

\usepackage{fullpage}
\usepackage{amsfonts}
\usepackage{enumitem}
\usepackage{amsmath}
\usepackage{array}
\usepackage{amssymb}
\usepackage{amsthm}
\usepackage{hyperref}
\usepackage{mathrsfs}
\usepackage{cite}
\usepackage{aliascnt}
\usepackage{centernot}
\usepackage{dsfont}
\usepackage{stmaryrd}
\usepackage{tikz-cd}
\usepackage{longtable}
\usepackage{rotating}
\DeclareMathOperator{\Irr}{Irr}

\DeclareMathOperator{\PSU}{PSU}

\DeclareMathOperator{\supp}{supp}

\begin{document}
\setlength{\parindent}{15pt}

\newtheorem{thm}{Theorem}[section]
\newcommand{\thmautorefname}{Theorem}

\numberwithin{equation}{thm}

\newtheorem*{thm*}{Theorem}

\newaliascnt{prop}{thm}
\newtheorem{prop}[prop]{Proposition}
\aliascntresetthe{prop}
\newcommand{\propautorefname}{Proposition}

\newaliascnt{lem}{thm}
\newtheorem{lem}[lem]{Lemma}
\aliascntresetthe{lem}
\newcommand{\lemautorefname}{Lemma}

\newaliascnt{crl}{thm}
\newtheorem{crl}[crl]{Corollary}
\aliascntresetthe{crl}
\newcommand{\crlautorefname}{Corollary}

\newtheorem{thmintro}{Theorem}
\renewcommand\thethmintro{\Alph{thmintro}}
\newcommand{\thmintroautorefname}{Theorem}

\newaliascnt{qstintro}{thmintro}
\newtheorem{qstintro}[qstintro]{Question}
\renewcommand\theqstintro{\Alph{qstintro}}
\aliascntresetthe{qstintro}
\newcommand{\qstintroautorefname}{Question}

\newaliascnt{conj}{thm}
\newtheorem{conj}[conj]{Conjecture}
\aliascntresetthe{conj}
\newcommand{\conjautorefname}{Conjecture}

\theoremstyle{definition}
\newaliascnt{dfn}{thm}
\newtheorem{dfn}[dfn]{Definition}
\aliascntresetthe{dfn}
\newcommand{\dfnautorefname}{Definition}
\newaliascnt{rmk}{thm}
\newtheorem{rmk}[rmk]{Remark}
\aliascntresetthe{rmk}
\newcommand{\rmkautorefname}{Remark}

\newaliascnt{exm}{thm}
\newtheorem{exm}[exm]{Example}
\aliascntresetthe{exm}
\newcommand{\exmautorefname}{Example}

\newaliascnt{qst}{thm}
\newtheorem{qst}[qst]{Question}
\aliascntresetthe{qst}
\newcommand{\qstautorefname}{Question}

\newtheorem*{qst*}{Question}

\title{Relations between values and zeros of irreducible characters of symmetric groups}
\author{Lee Tae Young\\\small Department of Mathematics and Statistics, Binghamton University, Binghamton, NY 13902, USA\\ \small tae.young@binghamton.edu\\ \small ORCID: 0000-0002-6849-3599}

\maketitle

\begin{abstract}
We prove certain polynomial relations between the values of complex irreducible characters of general finite symmetric groups. We use it to find some sets of conjugacy classes such that no finite symmetric group has a complex irreducible character that vanishes at every class in the set. In particular, we show that if $n$ satisfies certain conditions, then $S_n\setminus \{1\}$ cannot be covered by the set of zeros of three irreducible characters. We also prove that the values of character of $2$-defect zero can be expressed as rational functions in $n$, and build a recursive algorithm to find these rational functions. As another application, we improve a result by A. Miller on identification of irreducible characters by checking small number of values.
\end{abstract}



\tableofcontents

\section{Introduction}
Character tables of finite groups are known to have many zeros. A well-known theorem of Burnside says that every nonlinear irreducible complex character of every finite group has at least one zero. If an irreducible character has $p$-defect zero for some prime $p$, it vanishes at every element of order divisible by $p$. For symmetric groups, the Murnaghan-Nakayama rule tells us that whenever there is no border-strip tableaux for the paritions corresponding to the irreducible character and the conjugacy class, that character vanishes at that class. It is conjectured \cite{MS} that almost all zeros on the character tables of symmetric groups occur this way. Also, if the character corresponds to a self-conjugate partition, then it vanishes at all odd permutations.

However, it is not easy to locate or count zeros on the character table, unless we restrict to specific situations such as the ones mentioned above. For symmetric groups, although there are several formulas and algorithms that compute character values, such as Murnaghan-Nakayama rule and Frobenius formula, none of them gives any characterization of general zeros. 

In this paper, we are interested in the following question, in the case of symmetric groups.
\begin{qstintro}[G. Navarro\cite{Navarro}]\label{ThrChar}
For which finite groups $G$ there exist complex irreducible characters $\chi_1,\chi_2,\chi_3$ such that $\chi_1\chi_2\chi_3(g)=0$ for all $g\in G\setminus \{1\}$?
\end{qstintro}
There are several characters with many zeros, such as those of $p$-defect zero for a small prime $p$ and those corresponding to self-conjugate partitions, which vanishes much more than $1/3$ of the classes. However, to answer this question, we do not just need to count the zeros of irreducible characters, but also have to study how the zeros are distributed. 

Our first main result gives polynomial relations between values of irreducible characters of symmetric groups. This will be restated as \autoref{char_from_cycles} and \autoref{any_set}. It can also be viewed as a slight improvement of a result by Miller \cite[Theorem 3.8]{Miller} with a different proof; this is explained in \autoref{Mil}.
\begin{thmintro}\label{mainthm1}
Let $\lambda$ be a partition of a positive integer $m$. There exists a polynomial $T_\lambda$ of at most $m$ variables over the rational function field $\mathbb{Q}(x)$ such that for any complex irreducible character $\chi$ of the symmetric group $S_n$ and for an element $g\in S_n$ of cycle type $\lambda$, $\chi(g)/\chi(1)$ can be computed by plugging in $x=n$ and the values of $\chi/\chi(1)$ at cycles of length at most $m$ into the variables of the polynomial $T_\lambda$. 

This remains true if we replace each cycle by any other class that can be obtained as a composition of $k-1$ transpositions but not of any smaller number of transpositions, where $k$ is the length of the cycle.
\end{thmintro}
Using these polynomial relations and Gr\"obner bases, we can find many sets of classes such that no irreducible character of $S_n$ can vanish at every class in the set. In particular, we obtain a partial answer to \autoref{ThrChar}, which can be roughly stated as the following theorem; the precise statement will be given in \autoref{Zn>3}.

\begin{thmintro}\label{mainthm2}
If $n\equiv 2$ or $11$ mod $12$, $n\equiv 1$ or $4$ mod $5$, $n$ is large enough, and certain integers that depend on $n$ are not squares, then $S_n$ does not have three characters as in \autoref{ThrChar}.
\end{thmintro}

As an another application of \autoref{mainthm1}, we can compute the values of characters of $S_n$ of $2$-defect zero, this time only using the number $n$. This is our third main result; a precise statement is given in \autoref{staircase_poly_existence}.
\begin{thmintro}\label{mainthm3}
Let $\lambda$ be a partition of a positive integer $m$, and let $\psi_k$ be the irreducible character of $2$-defect $0$ of $S_{k(k+1)/2}$ for each positive integer $k$. There exists a rational function $R_\lambda(x)\in\mathbb{Q}(x)$ such that $R_\lambda(k(k+1)/2) = \psi_k(\lambda)/\psi_k(1)$. Also, there is a recursive algorithm computing the rational function $R_\lambda(x)$.
\end{thmintro}

In Section 2, we fix some notations, and prove \autoref{mainthm1}. In Section 3, we use the polynomial relations obtained from \autoref{mainthm1} to show various examples of the following form: if an irreducible character $\chi$ of $S_n$ vanishes at certain classes, then it is forced to be zero or to be nonzero at certain other classes, or $n$ is forced to satisfy certain conditions. Using the techniques developed in these examples, we prove \autoref{mainthm2} in Section 4. In section 5, we study the characters of $2$-defect $0$ and $3$-defect $0$, and prove \autoref{mainthm3}. Section 6 shows some examples which suggest that \autoref{mainthm1} might be sharp, in the sense that there might be no smaller set of conjugacy classes with the same properties.

Throughout this paper, we will mostly follow notations from the book \cite{Isaacs}. In particular, we will use the following notations. The set of all complex irreducible characters of a finite group $G$ will be denoted by $\Irr(G)$. The conjugacy class of $G$ containing $x\in G$ will be denoted by $x^G$. 

\section{Polynomial relations between values of irreducible characters}

Let us recall the following simple, well-known lemma. We included a proof for completeness.
\begin{lem}{\upshape\cite[Exercise 3.12]{Isaacs}}\label{general_poly_rel_thm}
Let $G$ be a finite group. Then for any $x,y\in G$ and any $\chi \in \Irr(G)$, \begin{equation}\label{general_poly_rel}|x^G|\chi(x)\chi(y^{-1}) = \chi(1)\sum_{x'\in x^G}\chi(x'y^{-1}).\end{equation}
\end{lem}
\begin{proof}
Let $\mathbb{C}G$ be the complex group algebra of $G$. For each $\chi\in\Irr(G)$, let $e_\chi$ be the corresponding \emph{primitive central idempotent} in $\mathbb{C}G$: $$e_\chi = \frac{\chi(1)}{|G|}\sum_{g\in G}\chi(g^{-1})g.$$ Let $x\in G$ be any element, $K=x^G$ be the conjugacy class of $G$ containing $x$, and $\hat{K}$ be its class sum. The coefficient of another class sum $\widehat{L}$ in the product $e_\chi\widehat{K}$ can be computed in two ways. First, one can write $$e_\chi\widehat{K} = \frac{\chi(1)}{|G|} \left(\sum_{g\in G}\chi(g^{-1})g\right)\widehat{K} = \frac{\chi(1)}{|G|}\sum_{g\in G}\sum_{x'\in K}\chi(g^{-1})gx' $$ so that the coefficient of $\widehat{L}$, which equals the coefficient of $y$, is 
\begin{equation}
\frac{\chi(1)}{|G|}\sum_{x'\in K}\chi(x'y^{-1}).
\end{equation} On the other hand, note that $e_\chi\widehat{K}$ is a scalar multiple of $e_\chi$, since the primitive central idempotents form a basis of $\mathbf{Z}(\mathbb{C}G)$, and $e_\chi e_\varphi = \delta_{\chi,\varphi}$ for any $\varphi\in \Irr(G)$. Since $$\chi(e_\chi \widehat{K}) = |K|\chi(x)\text{ and }\chi(e_\chi)=\chi(1),$$ it follows that $$e_\chi\widehat{K} = \frac{|K|\chi(x)}{\chi(1)}e_\chi = \frac{|K|\chi(x)}{|G|}\sum_{g\in G}\chi(g^{-1})g .$$ Comparing these two expressions gives \eqref{general_poly_rel}
\end{proof}

\begin{dfn}
For two elements $x,y$ of a finite group $G$, we define the polynomial \begin{equation}\label{general_poly_def}F_{x,y} = F_{x^G,y^G} := |x^G|t_{x^G}t_{(y^{-1})^G} - t_{\{1_G\}}\sum_{x'\in x^G}t_{(x'y^{-1})^G}\end{equation} in the polynomial ring $\mathbb{C}[t_{g^G} \mid g^G \text{ conjugacy class of }G]$. We will often write the variables as $t_g$ instead of $t_{g^G}$.
\end{dfn}

By \eqref{general_poly_rel}, every irreducible character $\chi$ of a finite group $G$ gives a solution of the system of equations $\{F_{x,y}=0 \mid x,y\in G\}$ by setting $t_x = \chi(x)$ for each $x\in G$. Since $F_{x,y}$ is homogeneous of degree $2$, every scalar multiple of $\chi$ also satisfies all $F_{x,y}$'s. We define \begin{equation}\rho_\chi(g) = \frac{\chi(g)}{\chi(1)}.\end{equation} So $\rho$ also satisfies the equations $F_{x,y}=0$. We will often write $\rho$ instead of $\rho_\chi$ if $\chi$ is obvious from the context.

Let $G=S_n$ be a finite symmetric group on the set $[n] := \{1,2,\dots, n\}$. Recall that the irreducible characters of $S_n$ only takes rational integers as its values, and every element of $S_n$ is conjugate to its own inverse. Hence, \eqref{general_poly_def} can be simplified to: \begin{equation}\label{symm_poly_rel}
F_{x,y} = |x^G|t_{x} t_{y} - t_{\{1\}}\sum_{x'\in x^G}t_{x'y}.
\end{equation}

Recall that the conjugacy classes of $S_n$ are determined by the cycle types of permutations, which then corresponds to partitions of $n$. We will use the notation $\lambda = (\lambda_1^{a_1},\lambda_2^{a_2}, \dots,\lambda_r^{a_r})$ to denote the partition with exactly $a_i$ parts of size $\lambda_i$ for each $i$. Given a partition $\lambda$ of $n$, we will abuse the notation and also denote by $\lambda$ the partitions of numbers $m\geq n$ obtained by adding parts of size $1$ to $\lambda$. We will often omit the parts of size $1$, so for example if we say $\lambda=(2)$, then it will mean $(2, 1^{n-2})$ for any $n$. We again abuse the notation and denote the class by the corresponding partition. We will also set $\chi(\lambda):=\chi(\sigma)$ for any element $\sigma$ of the class $\lambda$.

For a permutation $\sigma$ of cycle type $\lambda = (\lambda_1^{a_1},\dots,\lambda_r^{a_r})$, we define \begin{equation}\supp(\sigma)=\supp(\lambda) = \sum_{i=1}^{r}a_i\lambda_i\text{ and }\Vert \sigma\Vert = \Vert \lambda\Vert :=\sum_{i=1}^{r}a_i(\lambda_i-1)\label{def_transposition_number}.\end{equation} Note that $\Vert\sigma\Vert$ equals the minimum number of transpositions required to get $\sigma$ by composition; for example, $\Vert(3,2)\Vert= 3$, because $(3,2)$ can be obtained by composing $3$ transpositions.

We will use the ordering of partitions of a fixed number $n$ defined by the following rule. 
\begin{dfn}\label{part_ordering}
Let $\lambda=(\lambda_1,\lambda_2,\dots, \lambda_r)$ and $\sigma=(\sigma_1,\sigma_2,\dots,\sigma_s)$ be partitions where the parts are in decreasing order: $\lambda_1\geq \lambda_2\geq \cdots \geq \lambda_r$ and $\sigma_1\geq \sigma_2\geq \cdots \geq \sigma_s$. We define $\lambda>\sigma$ if $\Vert\lambda\Vert>\Vert\sigma\Vert$ (where $\Vert\lambda\Vert$ is as defined in \eqref{def_transposition_number}), or $\Vert\lambda\Vert=\Vert\sigma\Vert$ and there exists an $i$ such that $\lambda_j=\sigma_j$ for all $j<i$ and $\lambda_i<\sigma_i$. 
\end{dfn}
In other words, we first order by $\Vert \lambda\Vert$, then for those with the same value of $\Vert\lambda\Vert$, we use the ``reverse dominance ordering''. For example, $(2^3)<(5)<(4,2)<(3^2)<(3,2^2)$. This is a total ordering of the set of all partitions of $n$. This also defines an ordering of corresponding conjugacy classes of $S_n$, and of the variables $t_{K}$ appearing in \eqref{symm_poly_rel}. 

Our first result, which is the precise statement of \autoref{mainthm1}, shows that the value of any irreducible character of $S_n$ at any class can be expressed as a polynomial in its values at certain ``smaller'' classes, namely the cycles, where the coefficients are rational functions in $\mathbb{Q}(N)$. It also provides a recursive algorithm to compute these polynomials.
\begin{thm}\label{char_from_cycles}
Each irreducible character $\chi\in\Irr(S_n)$ can be computed from the positive integer $n$ and the ratios $\rho((r)) = \chi((r))/\chi(1)$, $2\leq r\leq n$. More precisely, for any partition $\lambda$ of $m\in \mathbb{Z}_{>0}$, there exists a polynomial $T_\lambda(t_{(2)},t_{(3)}, \dots, t_{(\Vert \lambda\Vert + 1)})$ over the field of rational functions $\mathbb{Q}(N)$, such that 
\begin{enumerate}[label={\upshape(\roman*)}]
\item for any  $n\geq m$ and $\chi\in \Irr(S_n)$, $\rho(\lambda)= T_\lambda(\rho((2)),\dots,\rho((\Vert \lambda\Vert+1)))(n)$ (so we plug the number $n$ into the variable $N$), and
\item if $\frac{p(N)}{q(N)}\prod_{i=1}^{\Vert \lambda\Vert}t_{(i+1)}^{b_i}$ appears as a term in $T_\lambda$, where $p(N),q(N)\in \mathbb{Q}[N]$ are relatively prime, then $\deg p\leq \deg q$, $\sum_{i=1}^{\Vert \lambda\Vert}b_ii \leq \Vert\lambda\Vert$, $\sum_{i=1}^{\Vert \lambda\Vert}b_ii \equiv \Vert\lambda\Vert$ mod $2$, and $\sum_{i=1}^{\Vert\lambda\Vert}b_i(i+1)\leq \supp(\lambda)$.
\end{enumerate}
\end{thm}
For convenience, we also define the polynomial $T$ for cycles as $T_{(k)} := t_{(k)}$.
\begin{proof}
We use induction on $\lambda$, using the ordering defined in \autoref{part_ordering}. There is nothing to prove if $\lambda=(\Vert \lambda\Vert+1)$. Suppose that $\lambda=(\lambda_1^{a_1}, \dots, \lambda_r^{a_r})$ with either $a_1>1$ or $r>1$, and assume that the statement holds for all classes smaller than $\lambda$ with respect to our ordering. Then $t_K$ is the largest variable appearing in the polynomial $F_{(\lambda_r), (\lambda_1^{a_1},\dots, \lambda_r^{a_r-1})}$. By \autoref{general_poly_rel_thm} applied to $F_{(\lambda_r),(\lambda_1^{a_1},\dots,\lambda_r^{a_r-1})}$, every $\chi\in \Irr(S_n)$ must satisfy
\begin{align*}
&\frac{\prod_{i=1}^{\lambda_r}(n-|\supp(\lambda)|+i)}{\lambda_r}\rho(\lambda) \\=&\frac{\prod_{i=0}^{\lambda_r-1}(n-i)}{\lambda_r} \rho((\lambda_r))\rho((\lambda_1^{a_1},\dots, \lambda_r^{a_r-1})) - \sum_{\substack{x\in (\lambda_r)\\xy\notin \lambda}}\rho(xy) \\=& \frac{\prod_{i=0}^{\lambda_r-1}(n-i)}{\lambda_r} \rho((\lambda_r))\rho((\lambda_1^{a_1},\dots, \lambda_r^{a_r-1})) - \sum_{\mu\neq \lambda}|\{x\in (\lambda_r)\mid xy\in \mu\}|\rho(\mu)
\end{align*}
where $\rho= \chi/\chi(1)$, and $y$ is a fixed element in the class $(\lambda_1^{a_1},\dots, \lambda_r^{a_{r-1}})$. In the last sum, $|\{x\in (\lambda_r)\mid xy\in \mu\}|$ is given by a polynomial in $n$ of degree at most $|\supp(\mu)|-|\supp(y)|= |\supp(\mu)|-|\supp(\lambda)|+\lambda_r<\lambda_r$, which is the $n$-degree of the left-hand side. The statement now follows by induction.
\end{proof}

\begin{rmk}
The main difficulty in the computation of the polynomials $F_{x,y}$ and $T_\lambda$ lies in counting the number of $x'\in x^{S_n}$ such that $x'y$ has cycle type $\mu$ for each partition $\mu$, which is often called the ``connection coefficients''. This can quickly become very complicated as the supports of the permutations $x$ and $y$ grow. When $x$ and $y$ are small enough, one method to compute this number is the following. First, one can compute it in $S_n$ for small, explicit $n$'s, namely $|\supp(y)|\leq n \leq |\supp(\mu)|$. One may use the following well-known formula (cf. \cite[Exercise 3.9]{Isaacs}) $$|\{x'\in x^{S_n}\mid x'y\in z^G\} |= \frac{|x^G||z^G|}{n!}\sum_{\chi\in\Irr(S_n)}\frac{\chi(x)\chi(y)\chi(z)}{\chi(1)}.$$ Then one can use Lagrange interpolation to find a general expression of this number as a polynomial in $n$. It would be nice if one can find a faster method to compute these numbers. Although there are some methods developed for several special cases, we are not aware of such method which is particularly good at giving the general polynomial formula.
\end{rmk}
\begin{rmk}\label{Remark2}
An anonymous person pointed out that it is also possible to use an explicit formula for character values by Lassalle \cite[Theorems 4 and 6]{Lassalle} together with Stirling inversion formula to find an explicit, closed formula for $T_\lambda$. This method might be faster in some cases than our inductive computation using $F_{x,y}$'s. However, due to the complexity of the formula given in \cite[Theorem 6]{Lassalle}, it seems to be not easy, if possible, to use this method to prove some properties of $T_\lambda$ we need, including those described in \autoref{char_from_cycles}(ii) and \autoref{leading_coeff}. Conversely, for explicit choices of conjugacy classes, \cite[Theorem 4]{Lassalle} combined with the polynomials $T_\lambda$ computed using our inductive method might give a simpler formula than the one given in \cite[Theorem 6]{Lassalle}, but it does not give a general formula that covers all classes at once.
\end{rmk}

\begin{exm}\label{example_trans_4}
Here we list the polynomials $T_\lambda$ for $\Vert \lambda\Vert\leq 4$. 
\begin{align*}
T_{(2^2)} :=& t_{(2^2)} + \frac{4t_{(3)}}{N - 3} + \frac{-N(N-1)t_{(2)}^2}{(N-2)(N-3)} + \frac{2}{(N-2)(N-3)}\\
T_{(3,2)} :=& t_{(3,2)} + \frac{6t_{(4)}}{N - 4} + \frac{-N(N-1)t_{(3)}t_{(2)}}{(N-3)(N-4)} + \frac{6t_{(2)}}{(N-3)(N-4)}\\
T_{(2^3)} :=& t_{(2^3)} + \frac{-40t_{(4)}}{(N-4)(N-5)} + \frac{12N(N-1)t_{(3)}t_{(2)}}{(N-3)(N-4)(N-5)} + \frac{ - N^2(N-1)^2t_{(2)}^3+(6N^2 - 70N + 120)t_{(2)}}{(N-2)(N-3)(N-4)(N-5)}\\
T_{(4,2)} :=& t_{(4,2)} + \frac{8t_{(5)}}{N - 5} + \frac{-N(N-1)t_{(4)}t_{(2)}}{(N-4)(N-5)} + \frac{8t_{(3)}}{(N-3)(N-4)} + \frac{4N(N-1)t_{(2)}^2-8}{(N-2)(N-3)(N-4)(N-5)}\\
T_{(3^2)} :=& t_{(3^2)} + \frac{9t_{(5)}}{N - 5} + \frac{-N(N-1)(N-2)t_{(3)}^2+(9N - 60)t_{(3)} }{(N-3)(N-4)(N-5)}\\&+ \frac{9N(N-1)t_{(2)}^2 +3(N - 8)}{(N-2)(N-3)(N-4)(N-5)} \\
T_{(3,2^2)}:=& t_{(3,2^2)} + \frac{-72t_{(5)}}{(N-5)(N-6)} + \frac{12N(N-1)t_{(4)}t_{(2)}}{(N-4)(N-5)(N-6)} \\&+ \frac{4N(N-1)(N-2)t_{(3)}^2 -N^2(N-1)^2t_{(3)}t_{(2)}^2 +2(N^2-61N+300)t_{(3)}  }{(N-3)(N-4)(N-5)(N-6)}\\& + \frac{12N(N-1)(N-8)t_{(2)}^2-24(N-8)}{(N-2)(N-3)(N-4)(N-5)(N-6)} \\
T_{(2^4)} :=& t_{(2^4)} + \frac{672t_{(5)}}{(N-5)(N-6)(N-7)} + \frac{-160N(N-1)t_{(4)}t_{(2)}}{(N-4)(N-5)(N-6)(N-7)} \\&+ \frac{-48N(N-1)(N-2)t_{(3)}^2+24N^2(N-1)^2t_{(3)}t_{(2)}^2-48(N^2-31N+140)t_{(3)}}{(N-3)(N-4)(N-5)(N-6)(N-7)} \\&+ \frac{-N^3(N-1)^3t_{(2)}^4+4N(N - 1) (3 N^2 - 67 N + 324)t_{(2)}^2-12(N^2-33N+196)}{(N-2)(N-3)(N-4)(N-5)(N-6)(N-7)} 
\end{align*}
\end{exm}

The proof of \autoref{char_from_cycles} suggests that many other sets of classes instead of the cycles can be also used to determine the character. For example, if a character has $p$-defect $0$, then it vanishes at every class of order divisible by $p$. Hence, it would be convenient if we can replace each cycle with a class of order divisible by $p$ which can be obtained as a composition of the same number of transpositions as the cycle. To replace $(k+1)$ with a class $\lambda$ with $\Vert \lambda\Vert= k$, we only need one condition: the polynomial $T_{\lambda}$ has a nonzero term involving $t_{(k+1)}$, which is automatically of total degree $1$ by \autoref{char_from_cycles}. The next proposition shows that this is always true.

\begin{thm}\label{leading_coeff}
Let $\lambda=(\lambda_1^{a_1},\dots,\lambda_r^{a_r})$ be a partition of some positive integer and let $k=\Vert \lambda\Vert > 0$. Let $T_{\lambda}$ be the polynomial defined in \autoref{char_from_cycles}. Then there exists exactly one term in $T_{\lambda}$ that involves $t_{(k+1)}$, and this term is $$(-1)^{1+\sum_{i=1}^{r}a_i}\frac{(\supp(\lambda)-1)!\left(\prod_{i=1}^{r}\lambda_i^{a_i}\right)}{(\Vert\lambda\Vert+1)!\prod_{i=k+1}^{\supp(\lambda)-1}(N-i)}t_{(k+1)}. $$
\end{thm}

\begin{proof}
We use induction on $\lambda$ with the ordering defined in \autoref{part_ordering}. The statement is obvious when $\lambda$ is already a cycle. Now suppose that the statement holds for all classes smaller than $\lambda$, and let $\lambda = (\lambda_1^{a_1},\dots,\lambda_r^{a_r})$ with $\lambda_r>1$ and either $r>1$ or $a_1>1$. Let $k=\Vert\lambda\Vert$. As we saw in the proof of \autoref{char_from_cycles}, for any permutation $y\in (\lambda_1^{a_1},\dots,\lambda_r^{a_r-1})$, we have 
\begin{align}\label{4.1_eqn}
&\frac{\prod_{i=1}^{\lambda_r}(n-\supp(\lambda)+i)}{\lambda_r}T_{\lambda}\nonumber \\=& \frac{\prod_{i=0}^{\lambda_r-1}(N-i)}{\lambda_r} t_{(\lambda_r)}T_{(\lambda_1^{a_1},\dots, \lambda_r^{a_r-1})} - \sum_{\mu\neq \lambda}|\{x\in (\lambda_r)\mid xy\in \mu\}|T_{\mu}.
\end{align}
Note that $T_{(\lambda_1^{a_1},\dots,\lambda_r^{a_r-1})}$ does not involve $t_{(k+1)}$ by \autoref{char_from_cycles}. Also, by the induction hypothesis, the only term in each $T_{\mu}$ appearing in the last sum that involves $t_{(k+1)}$ is of the form $ct_{(k+1)}/\prod_{i=k+1}^{\supp(\mu)-1}(N-i)$ for some $c\in \mathbb{Q}$, and by \autoref{char_from_cycles}(ii), they only appear when $\Vert\mu\Vert = k = \Vert \lambda \Vert$. In particular, there is only one term in $T_\lambda$ that involves $t_{(k+1)}$, which is of the form $t_{(k+1)}$ times a coefficient in $\mathbb{Q}(N)$. In the rest of this proof, we show that the coefficient is nonzero.

Fix a permutation $y\in (\lambda_1^{a_1},\dots,\lambda_r^{a_r-1})$, and let $x$ be a $\lambda_r$-cycle. We claim that $\Vert xy\Vert = k$ if and only if each disjoint cycle of $y$ intersects the support of $x$ at at most one point. Suppose that the support of $x$ intersects one of the disjoint cycles of $y$ at two or more points. We may write $x$ is the cycle $(x_1\ x_2\ \cdots \ x_{\lambda_r})$ and a disjoint cycle of $y$ as $(y_1\ y_2\ \cdots\ y_m)$ with $x_1=y_1$ and $x_i=y_j$ for some $1<i\leq \lambda_r$ and $1<j\leq m$. Then $xy = x(x_1\ x_i)(y_1\ y_j)y,$ so \begin{align*}
\Vert xy\Vert \leq \Vert x(x_1\ x_i)\Vert + \Vert (y_1\ y_j)y\Vert = \Vert x\Vert - 1 + \Vert y \Vert - 1 = \lambda_r-1 - 1 + k-(\lambda_r-1)-1 < k.\end{align*} Conversely, if each disjoint cycle of $y$ intersects the support of $x$ at at most one point, then $xy$ is the product of the disjoint cycles of $y$ which are disjoint from the support of $x$, and one additional cycle whose support is the union of the support of $x$ and the supports of the cycles of $y$ that intersect the support of $x$. Therefore $\Vert xy\Vert=k$.

For a class $\mu$ with $\Vert\mu\Vert = k$, set
\begin{align*}
m_j &:=  \max\left(a_j - \delta_{j,r} - \text{(number of parts of }\mu\text{ of size }\lambda_j), 0\right)\\&=\text{number of cycles of }y\text{ of length }\lambda_j\text{ that intersects the support of a }\lambda_r\text{-cycle }x\text{ with }xy\in \mu.
\end{align*} Then we can write
\begin{align*}
&\left|\{x\in(\lambda_r)\mid xy\in\mu\}\right|=\left(\prod_{j=1}^{r}{a_j-\delta_{j,r}\choose m_j}\lambda_j^{m_j} \right){n-\supp(\lambda)+\lambda_r\choose \lambda_r-\sum_{j=1}^{r}m_j}\frac{\lambda_r!}{\lambda_r}\\=&\left(\prod_{j=1}^{r}{a_j-\delta_{j,r}\choose m_j}\lambda_j^{m_j} \right)\frac{\lambda_r!}{\lambda_r}\frac{1}{(\lambda_r - \sum_{j=1}^{r}m_j)!}\prod_{i=1}^{\lambda_r - \sum_{j=1}^{r}{m_j}}(n-\supp(\lambda)+\lambda_r-i+1).
\end{align*}
Note that $\supp(\lambda)= \sum_{i=1}^{r}a_i\lambda_i$ and $\supp(\mu) = \sum_{i=1}^{r}a_i\lambda_i - \lambda_r + (\lambda_r-\sum_{j=1}^{r}m_i)$, so that $\sum_{j=1}^{r}m_j = \supp(\lambda) - \supp(\mu)$. Using this, we can divide $\left|\{x\in(\lambda_r)\mid xy\in\mu\}\right|$ by the coefficeint of $T_\lambda$ in the left-hand side of \eqref{4.1_eqn} and get
\begin{align*}
&\frac{\lambda_r\left|\{x\in(\lambda_r)\mid xy\in\mu\}\right|}{\prod_{i=1}^{\lambda_r}(N-\supp(\lambda)+i)} \\=& \left(\prod_{j=1}^{r}{a_j-\delta_{j,r}\choose m_j}\lambda_j^{m_j} \right)\frac{\lambda_r!}{\lambda_r}\frac{1}{(\lambda_r - \sum_{j=1}^{r}m_j)!}\frac{\lambda_r \prod_{i=1}^{\lambda_r - \sum_{j=1}^{r}{m_j}}(N-\supp(\lambda)+\lambda_r-i+1)}{\prod_{i=1}^{\lambda_r}(N-\supp(\lambda)+i)}\\=&\left(\prod_{j=1}^{r}{a_j-\delta_{j,r}\choose m_j}\lambda_j^{m_j} \right)\frac{1}{(\lambda_r - \sum_{j=1}^{r}m_j)!}\frac{\lambda_r !   }{\prod_{i=1}^{\supp(\lambda)-\supp(\mu)}(N-\supp(\lambda)+i)}.
\end{align*}
By this and the induction hypothesis, the only term of $\left|\{x\in(\lambda_r)\mid xy\in\mu\}\right|T_{\mu}$ that involves $t_{(k+1)}$ is  
\begin{align*}
&\frac{\left(\prod_{j=1}^{r}{a_j-\delta_{j,r}\choose m_j}\lambda_j^{m_j} \right)\lambda_r !  (\supp(\mu)-1)!\left(\prod_{i=1}^{r}\lambda_i^{a_i-\delta_{i,r}-m_i}\right)\left(\lambda_r + \sum_{i=1}^{r}m_j(\lambda_j-1)\right)t_{(k+1)}  }{(-1)^{1+\sum_{i=1}^{r}a_i-(\sum_{j=1}^{r}m_j-1)}(\lambda_r - \sum_{j=1}^{r}m_j)!\left(\prod_{i=1}^{\supp(\lambda)-\supp(\mu)}(N-\supp(\lambda)+i)\right)(k+1)!\prod_{i=k+1}^{\supp(\mu)-1}(N-i)}
\\=&\frac{\left(\prod_{j=1}^{r}{a_j-\delta_{j,r}\choose m_j} \right)\lambda_r !  (\sum_{i=1}^{r}a_i\lambda_i-\sum_{j=1}^{r}m_j -1)!\left(\lambda_r + \sum_{i=1}^{r}m_j(\lambda_j-1)\right)}{(-1)^{\sum_{i=1}^{r}a_i-\sum_{j=1}^{r}m_j} (\lambda_r - \sum_{j=1}^{r}m_j)!\lambda_r}   \frac{\left(\prod_{j=1}^{r}\lambda_j^{a_j} \right)t_{(k+1)}}{(k+1)!\prod_{i=k+1}^{\supp(\lambda)-1}(N-i)}
\end{align*}
Now it is enough to show that the sum of the coefficient of $ \frac{\left(\prod_{j=1}^{r}\lambda_j^{a_j} \right)t_{(k+1)}}{(k+1)!\prod_{i=k+1}^{\supp(\lambda)-1}(N-i)}$ in the above expression for all $\mu$ equals the one given in the statement of this proposition:
\begin{align*}
&\sum_{\substack{1\leq m_1+\cdots+m_r \leq \lambda_r\\m_i\in \{0,\dots, a_i-\delta_{i,r}\}}} \frac{\left(\prod_{j=1}^{r}{a_j-\delta_{j,r}\choose m_j} \right)\lambda_r !  (\sum_{i=1}^{r}a_i\lambda_i-\sum_{j=1}^{r}m_j -1)!\left(\lambda_r + \sum_{i=1}^{r}m_j(\lambda_j-1)\right)}{(-1)^{\sum_{i=1}^{r}a_i-\sum_{j=1}^{r}m_j} (\lambda_r - \sum_{j=1}^{r}m_j)!\lambda_r}   \\=& (-1)^{1+\sum_{i=1}^{r}a_i}\left(\sum_{i=1}^{r}a_i\lambda_i-1\right)!= (-1)^{1+\sum_{i=1}^{r}a_i}\left(\supp(\lambda)-1\right)!.
\end{align*}
Let $c_1, \dots, c_{\sum_{i=1}^r a_r -1}$ be the nontrivial disjoint cycles of $y$. Then we get
\begin{align*}
&\sum_{\substack{1\leq m_1+\cdots+m_r \leq \lambda_r\\m_i\in \{0,\dots, a_i-\delta_{i,r}\}}} \frac{\left(\prod_{j=1}^{r}{a_j-\delta_{j,r}\choose m_j} \right)\lambda_r !  (\sum_{i=1}^{r}a_i\lambda_i-\sum_{j=1}^{r}m_j -1)!\left(\lambda_r + \sum_{i=1}^{r}m_j(\lambda_j-1)\right)}{(-1)^{\sum_{i=1}^{r}a_i-\sum_{j=1}^{r}m_j} (\lambda_r - \sum_{j=1}^{r}m_j)!\lambda_r}\\=& \sum_{\substack{ I\subseteq \{1,\dots, \sum_{i=1}^{r}a_i - 1\}\\1\leq |I|\leq \lambda_r}} \frac{\lambda_r !  (\sum_{i=1}^{r}a_i\lambda_i-|I|-1)!\left(\lambda_r + \sum_{i\in I}(\supp(c_i)-1)\right)}{(-1)^{\sum_{i=1}^{r}a_i-|I|} (\lambda_r - |I|)!\lambda_r}\\=& \sum_{1\leq \ell \leq \lambda_r} \frac{\lambda_r !  (\sum_{i=1}^{r}a_i\lambda_i-\ell-1)!\left({\sum_{j=1}^{r}a_j-1\choose \ell}\lambda_r +{\sum_{j=1}^{r}a_j-2\choose \ell-1} \sum_{i=1}^{\sum_{j=1}^{r}a_j - 1}(\supp(c_i)-1)\right)}{(-1)^{\sum_{i=1}^{r}a_i-\ell} (\lambda_r - \ell)!\lambda_r}\\=& \sum_{1\leq \ell \leq \lambda_r} \frac{\lambda_r !  (\sum_{i=1}^{r}a_i\lambda_i-\ell-1)!\left({\sum_{j=1}^{r}a_j-1\choose \ell}\lambda_r +{\sum_{j=1}^{r}a_j-2\choose \ell-1}\left(\sum_{i=1}^{r}a_i\lambda_i - \lambda_r - (\sum_{j=1}^{r}a_j-1)\right)\right)}{(-1)^{\sum_{i=1}^{r}a_i-\ell} (\lambda_r - \ell)!\lambda_r}\\=& \sum_{1\leq \ell \leq \lambda_r} \frac{\lambda_r !  (\sum_{i=1}^{r}a_i\lambda_i-\ell-1)!\left({\sum_{j=1}^{r}a_j-1\choose \ell}(\lambda_r-\ell) +{\sum_{j=1}^{r}a_j-2\choose \ell-1}\left(\sum_{i=1}^{r}a_i\lambda_i - \lambda_r\right)\right)}{(-1)^{\sum_{i=1}^{r}a_i-\ell} (\lambda_r - \ell)!\lambda_r}\\=& \sum_{1\leq \ell \leq \lambda_r} \frac{\lambda_r !  (\sum_{i=1}^{r}a_i\lambda_i-\ell-1)!\left(\left({\sum_{j=1}^{r}a_j-1\choose \ell}-{\sum_{j=1}^{r}a_j-2\choose \ell-1}\right)(\lambda_r-\ell) +{\sum_{j=1}^{r}a_j-2\choose \ell-1}\left(\sum_{i=1}^{r}a_i\lambda_i - \ell\right)\right)}{(-1)^{\sum_{i=1}^{r}a_i-\ell} (\lambda_r - \ell)!\lambda_r}\\=& \sum_{1\leq \ell \leq \lambda_r-1}- \frac{\lambda_r !  {\sum_{j=1}^{r}a_j-2\choose \ell}(\sum_{i=1}^{r}a_i\lambda_i-(\ell+1))!}{(-1)^{\sum_{i=1}^{r}a_i-(\ell+1)} (\lambda_r - (\ell+1))!\lambda_r} +\frac{\lambda_r ! {\sum_{j=1}^{r}a_j-2\choose \ell-1}\left(\sum_{i=1}^{r}a_i\lambda_i - \ell\right)!}{(-1)^{\sum_{i=1}^{r}a_i-\ell} (\lambda_r - \ell)!\lambda_r}\\&+ \frac{\lambda_r ! {\sum_{j=1}^{r}a_j-2\choose \lambda_r-1}\left(\sum_{i=1}^{r}a_i\lambda_i - \lambda_r\right)!}{(-1)^{\sum_{i=1}^{r}a_i-\lambda_r}0!\lambda_r}\\=& (-1)^{1+\sum_{i=1}^{r}a_i}\left(\sum_{i=1}^{r}a_i\lambda_i - 1\right)!
\end{align*}
as claimed. 
\end{proof}

\begin{crl}\label{any_set}
Let $\{C_i\mid i\in \mathbb{Z}_{>0}\}$ be any set of classes such that for every $k$, $\Vert C_k\Vert = k$. Then there exist polynomials that satisfies \autoref{char_from_cycles}(i), so that for any $n$, $\chi\in\Irr(S_n)$ and any class $\lambda$, the value $\rho_\chi(\lambda)$ can be computed from the number $n$ and the values $\rho_\chi(C_k)$ for $1\leq k\leq \Vert\lambda\Vert$. 
\end{crl}

\begin{proof}
We use induction to show that there exist polynomials $\tilde{T}_{(k+1)}$ in the variables $t_{C_i}$ with coefficients in $\mathbb{Q}(n)$ such that for any sufficiently large $n$ and any $\chi\in\Irr(S_n)$, $\rho_\chi((k+1))=\tilde{T}_{(k+1)}(\rho_\chi(C_1),\dots,\rho_\chi(C_k))$. If $k=1$, then the only class $C_1$ with $\Vert C_1\Vert=1$ is $(2)$, so there is nothing to prove. Now suppose that for every class $\lambda$ with $\Vert\lambda\Vert<k$, there exists a polynomial $\tilde{T}_\lambda$ in variables $t_{C_1},\dots,t_{C_{\Vert \lambda\Vert}}$ such that $\tilde{T}_\lambda(\chi(C_1),\dots,\chi(C_{\Vert\lambda\Vert})) = \chi(\lambda)$ for all sufficiently large $n$ and all $\chi\in\Irr(S_n)$. By \autoref{leading_coeff} we have $$T_{C_k} = (-1)^{1+\sum_{i=1}^{r}a_i}\frac{(\supp(\lambda)-1)!\left(\prod_{i=1}^{r}\lambda_i^{a_i}\right)}{(\Vert\lambda\Vert+1)!\prod_{i=k+1}^{\supp(\lambda)-1}(N-i)}t_{(k+1)} + P_k$$ for some polynomial $P_{k}$ which only involves the variables $t_{(2)},\dots,t_{(k)}$. By the induction hypothesis, we can rewrite $P_{k}$ as a polynomial $\tilde{P}_{k}$ in variables $t_{C_1},\dots,t_{C_{k-1}}$ by replacing each $t_{(i)}$ with $\tilde{T}_{(i)}$. Then we can choose $$\tilde{T}_{(k+1)} = \frac{t_{C_k} - \tilde{P}_{k}}{(-1)^{1+\sum_{i=1}^{r}a_i}\frac{(\supp(\lambda)-1)!\left(\prod_{i=1}^{r}\lambda_i^{a_i}\right)}{(\Vert\lambda\Vert+1)!\prod_{i=k+1}^{\supp(\lambda)-1}(N-i)}}.$$ By induction, such polynomials exist for each $k\in \mathbb{Z}_{>0}$. Now by replacing each variable $t_{(k+1)}$ in the polynomials $T_{\lambda}$ with $\tilde{T}_{(k+1)}$, we get the polynomials $\tilde{T}_{\lambda}$ in the variables $t_{C_1},\dots,t_{C_{\Vert\lambda\Vert}}$ such that $\rho_\chi(\lambda) = \tilde{T}_\lambda(\rho_\chi(C_1),\dots,\rho_\chi(C_{\Vert\lambda\Vert}))$.
\end{proof}

\begin{rmk}\label{Mil}
In \cite{CP}, Chow and Paulhus showed that that the irreducible characters of $S_n$ can be identified by an algorithm that checks at most $O(n^{3/2})$ values of the character, and asked if there is a more efficient algorithm. Miller \cite[Theorem 3.6, Proposition 3.8]{Miller} improved this by finding an algorithm that requires checking only $n$ classes. Our result, \autoref{any_set}, provides many explicit choices of $n-1$ classes that can be used instead of Miller's $n$ classes. Moreover, Miller's algorithm makes different choices of conjugacy classes for different characters, and these choices cannot be made in advance without looking at the characters. \autoref{any_set} resolves this inconvenience: the sets of $n-1$ classes are explicitly chosen in a way that does not depend on the characters, so that all characters can be identified from their values at the same $n-1$ classes.
\end{rmk}

\begin{crl}\label{def0}
Let $p$ be a prime, and let $\chi\in\Irr(S_n)$ have $p$-defect $0$. Then for any choice of classes $C_k$ with $\Vert C_k\Vert=k$ for $k=1,\dots, p-2$, the values of $\rho_\chi$ can be computed from the number $n$ and the values $\rho_\chi(C_k)$.
\end{crl}

\begin{proof}
This follows from the previous Corollary by choosing $C_i=(p,i+1-p)$ for $i\geq p-1$, and using the fact that $\chi((p,i-p))=0$ for all $i\geq p$.
\end{proof}

In section 4, we will give more details about characters of $p$-defect $0$.

\section{Zeros of irreducible characters}
From the polynomial relations found using \autoref{char_from_cycles}, we can show that if an irreducible character vanishes at certain classes, then it is also forced to vanish at some other classes. 

\begin{prop}\label{odd_zeros}
Let $k$ be a positive integer and let $n\geq 2k+2$. Suppose that $\chi(\sigma)=0$ for all $\sigma$ such that $\Vert\sigma\Vert$ is an odd number less than $2k+1$. Then either $\chi(\lambda)$ for all $\lambda$ with $\Vert\lambda\Vert=2k+1$ or $\chi(\lambda)\neq 0$ for all $\lambda$ with $\Vert\lambda\Vert=2k+1$.
\end{prop}

\begin{proof}
Let $\lambda=(\lambda_1^{a_1},\dots,\lambda_r^{a_r})$ be a class with $\Vert\lambda\Vert=2k+1$ and $\supp(\lambda)\leq n$. Since $\chi(\sigma)=0$ for all $\sigma$ such that $\Vert\sigma\Vert$ is odd and less than $2k+1$, by \autoref{char_from_cycles} and \autoref{leading_coeff} we have $\rho_\chi(\lambda) = (1/f(n))\rho_\chi((2k+2))$ for a polynomial $f$ over $\mathbb{Q}(n)$ whose roots are less than $\supp(\lambda)-1<n$. Therefore, $\chi(\lambda)=0$ if and only if $\chi((2k+2))=0$. 
\end{proof}

\begin{prop}\label{forcing_2}
For the following ordered pairs $(C, N)$ of a set of classes $C$ and a number $N$ that depends on $n$, the following is statement holds: if $n$ is large enough, $N$ is not a square, and $\chi\in S_n$ vanishes at every element of $C$, then $\chi((2))=0$.\\
{\upshape(a)} $C=\{(2^2),(4),(2^3)\},\ N=8n-15$\\
{\upshape(b)} $C=\{(3),(3,2),(2^3)\},\ N=6n^2-30n+40$\\
{\upshape(c)} $C=\{(2^2),(4),(3,2)\},\ N=2n^2+22n-48$\\
{\upshape(d)} $C=\{(2^2),(3,2),(2^3)\},\ N=-10n^2+82n-120$\\
{\upshape(e)} $C=\{(3),(4),(2^3)\},\ N=6n^2-70n+120$\\
{\upshape(f)} $C=\{(4),(3,2),(5),(3^2)\},\ N=(-3n^3+9n^2+300n-576)/9$\\
{\upshape(g)} $C=\{(2^2),(3,2),(4,2),(3,2^2)\},\ N=2n^2+22n-48$\\
{\upshape(h)} $C=\{(2^2),(3,2),(4,2),(2^4)\},\ N=(38n^2+34n-192)/7$\\
Also, if $\chi$ vanishes at every element of one of the following $C$, then $\chi((2))=0$:\\
{\upshape(i)} $C=\{(3),(4),(3,2)\}$\\
{\upshape(j)} $C=\{(2^2),(3,2),(3,2^2),(2^4)\}$ (this also forces $\chi((4,2))=0$)\\
{\upshape(k)} $C=\{(3),(4),(3^2),(3,2^2)\}$.
\end{prop}

\begin{proof}
Consider the ideal generated by the polynomials $T_{\lambda}$ for all $\lambda$ with $\Vert\lambda\Vert \leq 4$ which is not a single cycle (these are exactly those appearing in \autoref{example_trans_4}), together with the monomials $\{t_{\lambda}\mid \lambda\in C\}$. In each case, the reduced Gr\"obner basis of this ideal, with respect to the lexicographic monomial ordering based on our ordering of partitions, contains the following polynomial (among others): \\
(a) $t_{(2)}(t_{(2)}^2 + (-32n + 60)/(n^2(n-1)^2))$\\
(b) $t_{(2)}(t_{(2)}^2 + (-6n^2 + 30n - 40)/(n^2(n-1)^2))$\\
(c) $t_{(2)}(t_{(2)}^2 + (-2n^2 - 22n + 48)/(n^2(n-1)^2))$\\
(d) $t_{(2)}(t_{(2)}^2 + (10n^2 - 82n + 120)/(n^2(n-1)^2))$\\
(e) $t_{(2)}(t_{(2)}^2 + (-6n^2 + 70n - 120)/(n^2(n-1)^2))$\\
(f) $t_{(2)}(t_{(2)}^2 + (3n^3 - 9n^2 - 300n + 576)/(9n^2(n-1)^2))$\\
(g) $t_{(2)}^2(t_{(2)}^2 + (-2n^2 - 22n + 48)/(n^2(n-1)^2))$\\
(h) $t_{(2)}^2(t_{(2)}^2 + (-38n^2 - 34n + 192)/(7n^2(n-1)^2))$\\
Since the values of $\rho_\chi=\chi/\chi(1)$ gives a solution of the generating polynomials of the ideals, it also gives a solution of this polynomial. Therefore, $\rho((2))$ is a solution of the above polynomial in each case. This forces either $\chi((2))=0$ or $(\rho((2))n(n-1))^2 = N$. Also, for the cases (i)-(k), the Gr\"obner basis contains $t_{(2)}$ or $t_{(2)}^2$, so $\chi((2))=0$, and in the case (j), it also includes $t_{(4,2)}$ so $\chi((4,2))=0$.
\end{proof}

In the opposite direction, we can prove that the irreducible characters of symmetric groups cannot vanish at all elements of certain sets of classes. Consider a set of $k$ conjugacy classes $\lambda_1,\dots,\lambda_k$ with $\Vert\lambda_i\Vert\leq k$ for every $i$. If $\chi(\lambda_i)=0$ for all $i$, then from \autoref{char_from_cycles}, we get a system of $k$ equations with $k+1$ variables $n, t_{(2)},t_{(3)}, \dots,t_{(k+1)}$ which has $n$ and $\rho_\chi((i))$'s as a rational solution. Unless some equations in the system are redundant, solving this system gives a polynomial relation of two numbers, namely $n$ and $t_{(i)}$ for some small $i$; usually $i=2$ or $3$. This forces $n$ to satisfy certain conditions, or sometimes even prohibits the existence of such $\chi$ for all $n$. 

For the rest of this section, we will show several examples that illustrate methods to study these bivariate polynomial relations. We begin with the simplest example.
\begin{prop}\label{2-3-22}
Let $n\geq 4$ and $\chi\in\Irr(S_n)$.\\
{\upshape (a)} At least one of $\chi((2)),\chi((3)),\chi((2^2))$ is nonzero. \\{\upshape (b)} If $\chi((3))=\chi((2^2))=0$, then $n(n-1)/2$ is a square. \\{\upshape (c)} If $\chi((2)) = \chi((3))= 0$, then $n\equiv 0$ or $1$ mod $3$.\\{\upshape (d)} If $\chi((2)) = \chi((2^2)) = 0$, then $n\equiv 0$ or $1$ mod $4$.
\end{prop}
\begin{proof}
By \autoref{general_poly_rel_thm}, for every $n\geq 3$ and every $\chi\in\Irr(S_n)$, choosing $t_{\sigma}=\rho_\chi(\sigma)=\chi(\sigma)/\chi(1)$ for each $\sigma\in S_n$ gives a solution of the polynomial 
$$T_{(2^2)} = t_{(2^2)} + \frac{4t_{(3)}}{N - 3} + \frac{-N(N-1)t_{(2)}^2}{(N-2)(N-3)} + \frac{2}{(N-2)(N-3)}$$ defined using \autoref{char_from_cycles}. If $\chi((3))=\chi((2^2))=0$, then we get $n(n-1)\rho((2))^2=2,$ and $\rho((2))$ is a rational number, so $n(n-1)/2$ is a square and $\chi((2))\neq 0$.

Recall that for any element $g$ of a finite group $G$ and any $\varphi\in \Irr(G)$, we have $|g^G|\varphi(g)/\varphi(1) \in \mathbb{Z}$. If $\chi((2))=\chi((2^2))=0$, then we get $\frac{4\rho((3))}{n-3} = -\frac{2}{(n-2)(n-3)}$, so $$\frac{|(3)|\chi((3))}{\chi(1)} = \frac{n(n-1)(n-2)}{3}\rho((3)) = \frac{n(n-1)}{6}\in \mathbb{Z},$$ so $n\equiv 0$ or $1$ mod $3$. Similarly, if $\chi((2))=\chi((3))=0$, then we get $\rho((2^2)) = \frac{2}{(n-2)(n-3)}$, so $$\frac{|(2^2)|\chi(2^2)}{\chi(1)} = \frac{n(n-1)(n-2)(n-3)}{8}\rho((2^2)) = \frac{n(n-1)}{4}\in \mathbb{Z}.$$ Therefore $n\equiv 0$ or $1$ mod $4$.
\end{proof}
Note that case (b) in the above theorem happens quite rarely; the first few $n$'s such that $n(n-1)/2$ is a square are $1, 2, 9, 50, 289, 1682, \dots$, which appears as the sequence A055997 on \cite{OEIS}. These numbers take the form $(2 + (3+2\sqrt{2})^{k-1} + (3-2\sqrt{2})^{k-1})/4$ for positive integers $k$. $S_n$ for $n=9$ and $50$ actually have such characters, for example those corresponding to the partitions $(5,2^2)$ and $(14, 6^4, 5, 3^2, 1)$ vanishes at both $(2^2)$ and $(3)$. I do not know whether $S_n$'s for larger such $n$'s always have such irreducible characters.

To handle more complicated polynomial relations, we use Siegel's theorem on integral points of algebraic curves:
\begin{thm*}[\upshape Siegel \cite{Siegel}]
If an affine algebraic curve defined over a number field has nonzero genus or has more than two points at infinity, then there are only finitely many integral points on it.
\end{thm*}

Another fact we can utilize is that for any class $\lambda$ of $S_n$ and any $\chi\in\Irr(S_n)$, the number $$\omega_\chi(\lambda) := \frac{|\lambda|\chi(\lambda)}{\chi(1)}$$ is an algebraic integer; since the values of $\chi\in\Irr(S_n)$ are (rational) integers, $\omega_\chi(\lambda)\in \mathbb{Z}$. 

\begin{prop}\label{4-trans-forbidden}
There exists some integer $M>0$ such that if $n>M$ and $\chi\in\Irr(S_n)$ satisfies $\chi((5))=\chi((3^2))=0$, then $\chi((2)),\chi((2^2)),\chi((3))$ are nonzero.
\end{prop}
\begin{proof}
As in the proof of \autoref{forcing_2}, consider the ideal generated by the polynomials $T_{\lambda}$ for all $\lambda$ with $\Vert\lambda\Vert \leq 4$ which is not a single cycle, together with $t_{(5)}$ and $t_{(3^2)}$. The reduced Gr\"obner basis of this ideal includes the following polynomials:
\begin{align}
&t_{(2^2)} + \frac{4}{N - 3}t_{(3)} + \frac{-N(N-1)}{(N-2)(N-3)}t_{(2)}^2 + \frac{2}{(N-2)(N-3)}\label{x22_when_x5=x33=0}\\
&t_{(3)}^2 + \frac{-9N + 60}{N(N-1)(N-2)}t_{(3)} - \frac{9}{(N-2)^2}t_{(2)}^2 + \frac{-3(N-8)}{N(N-1)(N-2)^2}\label{x3_when_x5=x33=0}
\end{align}
Again, the values of $\rho_\chi$ form a solution of both \eqref{x22_when_x5=x33=0} and \eqref{x3_when_x5=x33=0}. 

First assume $\chi((3))=0$. Then \eqref{x3_when_x5=x33=0} with $N=n$, $t_{(3)}=\rho_\chi((3))=0$ and $t_{(2)}=\rho_\chi((2))$ becomes $$\frac{9}{(n-2)^2}\rho((2))^2 + \frac{3(n-8)}{n(n-1)(n-2)^2}=0.$$ Since $\rho((2))^2\geq 0$, this equality cannot hold unless $n\leq 8$.

Next, assume $\chi((2))=0$. Then \eqref{x3_when_x5=x33=0} becomes $$\rho((3))^2 + \frac{-9n + 60}{n(n-1)(n-2)}\rho((3)) + \frac{-3(n-8)}{n(n-1)(n-2)^2}=0.$$ Note that $\omega((3)) := |(3)|\rho((3)) = n(n-1)(n-2)\rho((3))/3$ is an integer. Hence, the pair $(x,y)=(n,\omega((3)))$ is an integer solution of the polynomial \begin{equation}\label{x2=x5=x33=0}y^2 + (-3x + 20)y - \frac{x(x-1)(x-8)}{3}=0.\end{equation} The algebraic curve defined by \eqref{x2=x5=x33=0} has genus $1$. By Siegel's Theorem, \eqref{x2=x5=x33=0} has only finitely many integer solutions, so there exists an upper bound for $n$.

Finally, assume $\chi((2^2))=0$. Then from \eqref{x22_when_x5=x33=0} we get $$\frac{4}{n - 3}t_{(3)} = \frac{n(n-1)}{(n-2)(n-3)}t_{(2)}^2 - \frac{2}{(n-2)(n-3)}.$$ Using this, we can rewrite \eqref{x3_when_x5=x33=0} as $$t_{(2)}^4 + \frac{-4n^2 - 32n + 96}{n^2(n-1)^2}t_{(2)}^2 + \frac{4(n^2 + 5n - 24)}{n^3(n-1)^3}=0.$$ As in the previous case, we can see that $(n, \omega((2)))$ is an integer point of an algebraic curve of genus $3$, so there exists an upper bound for $n$.
\end{proof}

\begin{prop}\label{2-5-42-odd}
Let $n\geq 5$. If $\chi\in\Irr(S_n)$ vanishes at $(2), (5), (4,2)$, and at least one (hence all) of $(4),(3,2),(2^3)$ is $0$, then $n\in \{7, 15, 25\}$. 
\end{prop}

\begin{proof}
The Gr\"obner basis computation as above shows that $\rho_\chi((3))=1/((n-2)(n-5))$. Note that $$\omega_\chi((3)) = |(3)|\rho_\chi((3)) = \frac{n(n-1)}{3(n-5)}\in \mathbb{Z}.$$ In particular, $n-5$ divides $n(n-1)$, so $n-5$ cannot have any prime factors other than $2$ and $5$. Moreover, $n-5$ cannot be divisible by $8$ or $25$. Therefore $n-5 = 2^a5^b$ for $a\in \{0,1,2\}$ and $b\in \{0,1\}$, so $n\in \{6, 7, 9, 10, 15, 25\}$. A computer search shows that among these values of $n$, only $7$, $15$ and $25$ have such $\chi$.
\end{proof}

\begin{prop}
There exists some integer $M>0$ such that if $n>M$, then there is no $\chi\in\Irr(S_n)$ vanishing at all of $(3,2),(3^2),(3,2^2),(2^4)$.
\end{prop}

\begin{proof}
Suppose that $\chi\in\Irr(S_n)$ satisfies $\chi((3,2))=\chi((3^2))=\chi((3,2^2))=\chi((2^4))=0$. Then a Gr\"obner basis computation as above shows that $\rho_\chi((2))$ must be a solution of the following polynomial: 
\begin{align*}
\left(t_{(2)}^2 + \frac{-11n^2 + 245n - 1350}{2n^2(n-1)^2}\right)\left(t_{(2)}^4 + \frac{-12n^2 + 108n - 304}{n^2(n-1)^2}t_{(2)}^2 + \frac{12n^2 - 172n + 560}{n^3(n-1)^3}\right)
\end{align*}

First suppose that $\rho_\chi((2))$ is a solution of the first factor. Then $$\frac{11n^2 - 245n + 1350}{8} = \frac{11n^2 - 245n + 1350}{2n^2(n-1)^2}\left(\frac{n(n-1)}{2}\right)^2 = \left(\frac{n(n-1)}{2}\right)^2\rho_\chi((2))^2=\omega((2))^2$$ is an integer, since $\omega((2)) = |(2)|\rho_\chi((2))=n(n-1)\rho_\chi((2))/2\in \mathbb{Z}$. Therefore $22n^2-490n+2700=2(n-10)(11n-135)$ is a square divisible by $16$. Assume that $n>12$. Note that $\gcd(n-10, 11n-135)$ divides $25$. It follows that $n-10$ must be a square times $1$ or $5$, and similarly $11n-135$ must be a square time $1$ or $5$. We have the following cases: \\
$\bullet$ If $n$ is odd and $n-10$ is a square, then $n-10\equiv 1$ mod $8$ and $11n-135=11(n-10)-25\equiv 2$ mod $8$, so $2(n-10)(11n-135)$ is not divisible by $16$, which is a contradiction. \\
$\bullet$ If $n$ is odd and $n-10$ is $5$ times a square, then $n-10\equiv 5$ mod $8$ and $11n-135\equiv 6$ mod $8$, so $2(n-10)(11n-135)$ is not divisible by $16$, which is a contradiction.\\
$\bullet$ if $n$ is even, then $n-10\equiv 0$ mod $8$, so $11n-135 \equiv 7$ mod $8$; however, since $11n-135$ is a square times $1$ or $5$, $11n-135\equiv 1$ or $5$ mod $8$, so we get a contradiction.\\
Therefore $n\leq 12$. We can easily check the remaining cases to see that this only happens when $n=10$.

Next, suppose that $\rho_\chi((2))$ is a solution of the second factor. By arguing as in the proof of \autoref{4-trans-forbidden}, we get an algebraic curve of genus $3$. By Siegel's theorem, there are only finitely many integer points $(n,\omega_\chi((2)))$ on that curve, so $n\leq M$ for some large integer $M$.
\end{proof}

Gr\"obner basis computations followed by the methods we have been using so far, namely Siegel's theorem, modular arithmetic, quadratic residues, and the fact that $\omega_\chi(\lambda) = \frac{|\lambda|\chi(\lambda)}{\chi(1)}$ is an integer for all $\chi\in\Irr(S_n)$, can be applied to find many other ``forbidden'' sets of zeros. If we restrict ourselves to the $11$ nonidentity classes which can be obtained as a composition of at most $4$ transpositions, we have $330$ subsets of size $4$. Among these, the subsets on which an irreducible character can vanish are the following.

\begin{thm}\label{4_trans-allowed}
Let $C\subseteq \{\lambda\mid 0<\Vert\lambda\Vert\leq 4\}$ be a set of classes with $|C|\geq 4$. Suppose that there exist infinitely many $n$'s with a character $\chi\in\Irr(S_n)$ vanishing on $C$. Then at least one of the following holds for $C$ and all such $n, \chi$:\\
{\upshape(i)} $|\{(2),(3),(2^2)\}\cap C|= 2$, so that $n$ satisfies one of the conditions in \autoref{2-3-22}. \\
{\upshape(ii)} One of the following is a square of an integer: $(6n^2 + 74n - 600)/4,\ (38n^2 + 34n - 192)/28,\ 8n - 45,\ 8n-15,\ (6n^2 - 30n + 40)/4,\  (6n^2 - 70n + 120)/4,\ (11n^2 - 245n + 1350)/8,\ (2n^2 + 22n - 48)/4,\ (6n^2 + 2n - 24)/4,\ (n^2 - n)/2 $.\\
{\upshape(iii)} $\rho_\chi((3)) = (n^2 - 25n + 60)/(2n(n-1)(n-2))$; in this case $n\equiv 0$ or $1$ mod $3$.\\
{\upshape(iv)} $\rho_\chi((3)) = -(n^2 - 33n + 140)/(2n(n-1)(n-2))$; in this case $n\equiv 1$ or $2$ mod $3$ and $C=\{(2),(4),(3,2),(2^3),(4,2),(2^4)\}$.\\
{\upshape(v)} $\rho_\chi((3))= 12(n-5)/(n(n-1)(n-2))$; in this case $C=\{(2),(4),(3,2),(2^3),(4,2),(3,2^2)\}$, $\rho_\chi((2^2)) = -(2n^2 +46n-240)/(n(n-1)(n-2)(n-3)$, and $n\equiv 0$ or $1$ mod $4$ .\\
{\upshape(vi)} $\rho_\chi((3)) = -(9n^2 - 129n + 420)/(4n(n-1)(n-2))$; in this case $n\equiv 0$ or $1$ mod $4$.\\
{\upshape(vii)} $\{(2),(4),(3,2),(2^3)\}\subseteq C$ and $C\cap \{(5),(4,2),(3^2),(3,2^2),(2^4)\}\leq 1$.\\
{\upshape(viii)} $C=\{(3),(3,2),(3^2),(3,2^2)\}$.\\
{\upshape(ix)} $C=\{(2^2),(4),(4,2),(3,2^2)\}$.
\end{thm}

\begin{proof}
Using the above methods and Magma, we checked all subsets of size at least $4$ of $\{(2), (3),(2^2),(4),(3,2),(2^3), (5),(4,2), (3,3),(3,2^2),(2^4)\}$. It turned out that each subset $C$ which is not listed above in (i), (vii), (viii) and (ix) satisfies at least one of the following conditions:
\begin{enumerate}[label=(\alph*)]
\item The Gr\"obner basis is $\{1\}$.
\item The Gr\"obner basis contains a polynomial over $\mathbb{Q}(n)$ in some single variable $t$ where each irreducible factor is either of the form $t^2 - f(n)/g(n)$ for some variable $t$ and some polynomials $f,g\in \mathbb{Q}[n]$ where the leading coefficient of $f$ is negative and $g$ is monic, or of the form $F(|t|t,n)$ for some polynomial $F\in \mathbb{Q}[t,n]$ which defines an affine curve of nonzero genus, where $|t|$ is the size of the conjugacy class corresponding to $t$, which is a polynomial in one variable $n$.
\item The Gr\"obner basis contains a polynomial over $\mathbb{Q}(n)$ in some single variable $t$ which has a factor of the form $t^2 - f(n)g(n)$ for the variable $t=t_{(2)}$ or $t_{(3)}$, where $f$ is one of the polynomials listed in (ii) and $g$ is a square in $\mathbb{Q}(n)$.
\item The Gr\"obner basis contains $t_{(3)} - (n^2-25n+60)/(2n(n-1)(n-2))$ or $t_{(3)} + (9n^2 - 129n+420)/(4n(n-1)(n-2))$.
\item The Gr\"obner basis contains some monomial of the form $t_{\lambda}^k$ for some $\lambda\notin C$. In this case, every irreducible character vanishing on $C$ also vanishes at $\lambda$.
\end{enumerate}
See \autoref{forbiddentable} for the list of subsets $C$ of size $4$ in cases (i), (a) and (b); there are $208$ such sets.

(a) and (b) imply that for only finitely many $n$'s, $S_n$ has an irreducible character vanishing on $D$ for any $D\supseteq C$. (c) implies (ii), and (d) implies (iii) or (vi). For (e), we can add all such $\lambda$'s to $C$ to get larger sets. Again using Magma, we checked that these sets satisfy either (a) or (b) except for those listed in (i), (iv), (v), and (vii). 
\end{proof}

\begin{crl}\label{4_trans-special}
In the situation of \autoref{4_trans-allowed}, if $n\equiv 2$ or $11$ mod $12$, $n$ is large enough, none of the numbers in (ii) is a square, and there exists $\chi\in \Irr(S_n)$ which vanishes on $C$, then $C$ is as described in (iv), (vii),(viii) and (ix).
\end{crl}

\section{Covering symmetric groups with zeros of characters}

We would like to study the following questions, which are reformulations of \autoref{ThrChar}.
\begin{qst}
(1) For each $n\in \mathbb{Z}_{>0}$, what is the smallest number $Z(n)$ such that there exists $\chi_1,\dots,\chi_{Z(n)}\in \Irr(S_n)$ such that $\prod_{i=1}^{Z(n)}\chi_i(\sigma)=0$ for all $\sigma\in S_n\setminus \{1\}$, or equivalently $\bigcup_{i=1}^{Z(n)}\{\sigma\in S_n\mid \chi_i(\sigma)=0\} = S_n\setminus \{1\}$?\\
(2) For integers $n>k>0$, what is the smallest number $Z_k(n)$ such that there exists $\chi_1,\dots,\chi_{Z_k(n)}\in \Irr(S_n)$ such that $\prod_{i=1}^{Z(n)}\chi_i(\sigma)=0$ for all $\sigma$ with $1\leq \Vert \sigma\Vert \leq k$? \\(Obviously $Z(n)=Z_{n-1}(n)\geq Z_{n-2}(n)\geq \cdots \geq Z_1(n)$.)
\end{qst}

Based on the discussion above \autoref{2-3-22}, we can expect most irreducible characters to have $k$ or less zeros among the conjugacy classes $\lambda$ with $\Vert \lambda\Vert\leq k$, with some exceptions having more zeros, most notably those with $p$-defect $0$ for some prime $p$ and those corresponding to self-conjugate Young diagrams. The zeros coming from the self-conjugacy of Young diagrams are the odd permutations, so if we restrict ourselves to the even permutations, self-conjugacy does not provide any obvious zeros. Also, Olsson and Stanton \cite[Theorem 4.1]{OS} proved that there exists no irreducible character of $S_n$ which simultaneously has $p$-defect and $q$-defect $0$ for two distinct primes $p,q$ unless $n\leq (p^2-1)(q^2-1)/24$. Therefore, choosing $p$-defect zero characters for each prime $p$ might be the most efficient way to cover many classes by zero sets.

On the other hand, the asymptotic formula for the number of partitions of $n$ by Hardy and Ramanujan says $$\text{Number of conjugacy classes of }S_n \sim \frac{1}{4n\sqrt{3}}e^{\pi\sqrt{2n/3}}.$$ In particular, as $k$ increases, the number of classes $\lambda$ with $\Vert\lambda\Vert\leq k$ grows much faster than the expected number of zeros among these an irreducible character can have. 

Based on these observations, it is natural to make the following guesses.
\begin{conj}\label{conjectures}
{\upshape(1)} For any $N\in \mathbb{Z}_{>0}$, there exists a positive integer $M$ such that for all $n>M$, $Z(n)>N$.\\
{\upshape(2)} If $n\notin \{5,6,8,9,10,12,21\}$, then $Z(n)>3$.\\
{\upshape(3)} If $n>>k>>0$, then $Z_k(n)$ is at least the number of distinct primes not exceeding $k+1$.
\end{conj}
\autoref{conjectures}(2) is based on our computation for $n\leq 90$; the only $n$'s with $Z(n)=3$ are $5,6,8,9,10,12,21$. For \autoref{conjectures}(3), we know that $Z_2(n)\geq 2$ for all $n\geq 4$ by \autoref{2-3-22}.

Using the methods in Section 3, we can prove the following partial result for \autoref{conjectures}, which is a precise restatement of \autoref{mainthm2}.
\begin{thm}\label{Zn>3}
If $n\equiv 2$ or $11$ mod $12$, $n\equiv 1$ or $4$ mod $5$, $n$ is large enough, and none of the numbers in \autoref{4_trans-allowed}(ii) is a square, then $Z_6(n)>3$.
\end{thm}

\begin{proof}
Let $n$ be a positive integer satisfying the conditions. Suppose that $Z_6(n)=3$, so that there exists $\chi_1,\chi_2,\chi_3\in\Irr(S_n)$ such that at each $\lambda$, at least one of $\chi_i$'s vanish. Let $V_i = \{\lambda \mid \chi_i(\lambda)=0\}$. We will look at the sets $L_k:=\{\lambda \mid \Vert\lambda\Vert=k\}$ and $L_{\leq k}:=\{\lambda\mid 1\leq \Vert\lambda\Vert \leq k\}$ for each small $k$ and find all possible combinations of $V_i$'s that cover these sets.\newline

(I) $k=1, 2$.\\ 
By \autoref{2-3-22}, none of $\chi_i$ can vanish at more than one of the classes $(2), (3), (2^2)$. So we may assume that $\chi_1((2))=\chi_2((3))=\chi_3(2^2)=0$.\newline

(II) $k=3$.\\
By \autoref{forcing_2} (a)-(e) and (i), $\chi_2$ and $\chi_3$ cannot vanish at more than one of the classes $(4),(3,2),(2^3)$. Note that the case \autoref{forcing_2}(d) does not appear in \autoref{4_trans-allowed}(ii), but when $n\geq 7$ the number $-10n^2+82n-120$ is negative, so it cannot be a square, so this case is covered by the assumption that $n$ is large enough. Since we have three classes $(4),(3,2),(2^3)$ and each of $\chi_2,\chi_3$ can cover at most one, $\chi_1$ should also vanish at at least one of these classes. By \autoref{odd_zeros}, $\{(4),(3,2),(2^3)\}\subset V_1$. \newline

(III) $k=4$.\\
We have $5$ classes to cover here: $L_4=\{(5), (4,2), (3^2), (3,2^2), (2^4)\}$. By \autoref{4_trans-special}, the only case when $\chi_1$ vanishes at more than $1$ of them is when $L_4\cap V_1=\{(4,2),(2^4)\}$. Also, by \autoref{4_trans-allowed}, each of $\chi_2$ and $\chi_3$ cannot vanish at more than $2$ of these classes. Therefore, $\chi_1$ vanishes at either exactly $1$ of these classes or at $(4,2)$ and $(2^4)$, at least one of $\chi_2,\chi_3$ must vanish at exactly $2$ of these classes, and the remaining one must vanish at all of the remaining (one or two) classes. Also, $\{(5),(3^2)\}\not\subset V_i$ for each $i$ by \autoref{4-trans-forbidden}. \newline

(IV) $k=5$.\\
Suppose that $V_1\cap L_5=\emptyset$. Since $L_5 = \{(6),(5,2),(4,3),(4,2^2),(3^2,2),(3,2^3),(2^5)\}$ has $7$ classes, at least one of $V_2$ and $V_3$ must contain $4$ or more of these. We checked each of the possible $V_2\cap L_{\leq 5}$ and $V_3\cap L_{\leq 5}$, namely the subsets of $L_{\leq 5}$ satisfying the following conditions:
\begin{itemize}
\item It contains exactly one of $(3)$ and $(2^2)$;
\item It does not contain $(2)$;
\item It contains at least one element from $L_4$; 
\item It contains $4$ elements from $L_5$.
\end{itemize}
Each of these sets gives a Gr\"obner basis which contains a polynomial whose only irreducible factors are among the following 9 polynomials: $t_{(2)}, t_{(2)}^2 + (n-8)/(3n(n-1)), t_{(2)}^2 - 8(n-5)/(3n(n-1)(n-6)), t_{(2)}^2 -(2 (11 n^3 - 177 n^2 + 706 n - 840))/(3 (n^2 (n - 10) (n - 1)^2)),t_{(2)}^4 + (-20n^2 + 260n - 1120)/(n^2(n-1)^2)t_{(2)}^2 + (60n^4 - 1400n^3 + 11852n^2 - 37392n + 40320)/(n^4(n-1)^4), t_{(2)}^2 + (n^3 - 39n^2 + 188n - 240)/(3n^2(n-1)^2), t_{(2)}^4 + (-12n^2 + 108n - 304)/(n^2(n-1)^2)t_{(2)}^2 + (12n^2 - 172n + 560)/(n^3(n-1)^3),t_{(2)}^2 + (-32n + 60)/(n^2(n-1)^2), t_{(2)}^2 - 2/(n(n-1))$ unless the subset is $\{(3),(3,2^2),(6),(4,3),(3^2,2),(3,2^3)\}$ or $\{(3),(3^2),(6),(4,3),(3^2,2),(3,2^3)\}$. 

Among these $9$ polynomials, $t_{(2)}$ and $t_{(2)}^2 - 2/(n(n-1))$ cannot be zero at $t_{(2)}=\rho_{\chi_i}((2))$, since that would force $|V_i\cap L_{\leq 2}|\geq 2$. The polynomials $t_{(2)}^2 + (n-8)/(3n(n-1))$ and $t_{(2)}^2 + (n^3 - 39n^2 + 188n - 240)/(3n^2(n-1)^2)$ cannot be zero at $t_{(2)}=\rho_{\chi_i}((2))$ if $n\geq 8$ and $n\geq 34$, respectively, since $\rho_{chi_i}((2))> 0$. If $t_{(2)}^2 - 8(n-5)/(3n(n-1)(n-6))$ is zero at $t_{(2)}=\rho_{\chi_i}((2))$, then as in \autoref{2-5-42-odd}, we have $\frac{2n(n-1)(n-5)}{3(n-6)}\in \mathbb{Z}$, so $n-6$ cannot have prime factors other than $2,3,5$, and is not divisible by $8, 9$ and $25$, hence $n\leq 30$. The polynomial $t_{(2)}^2 + (-32n + 60)/(n^2(n-1)^2)$ is nonzero at $t_{(2)}=\rho_{\chi_i}((2))$ by the assumption that $8n-15$ (in \autoref{4_trans-allowed}(ii)) is not a square. The remaining three polynomials define affines curve of nonzero genus, so by Siegel's theorem, they cannot be zero at $t_{(2)}=\rho_{\chi_i}((2))$ if $n$ is large enough. 

It remains to check the cases $V_2\cap L_5 = \{(6),(4,3),(3^2,2),(3,2^3)\}$. In this case $V_3\supseteq \{(2^2), (5,2), (4,2^2), (2^5)\}$ and $|V_3\cap L_4|\geq 1$. Each of these sets gives a Gr\"obner basis containing a polynomial of the form $t_{(2)}f(n,n(n-1)t_{(2)}/2)$, where $f$ is a bivariate polynomial which defines an affine curve of nonzero genus. By Siegel's theorem, this cannot be zero at $t_{(2)}=\rho_{\chi_i}((2))$ if $n$ is large enough. Therefore, $V_2\cup V_3$ can never contain $L_5$, so $V_1\cap L_5$ cannot be empty, hence by \autoref{odd_zeros}, $L_5\subset V_1$.\newline

(V) $k=6$. \\
Here we have $11$ classes: $L_6 = \{(7),(6,2),(5,3),(5,2^2),(4^2),(4,3,2),(4,2^3),(3^3),(3^2,2^2),(3,2^4),(2^6)\}$. Suppose that $|V_1\cap L_6|\leq 2$. Then at least one of $V_2, V_3$ must contain at least $5$ elements of $L_6$. We checked using Magma that every subset of $L_{\leq 6}$ that contains exactly one of $(2),(3),(2^2)$, at least one element from $L_4$ and at least $5$ elements of $L_6$ has Gr\"obner basis with a polynomial that either defines an affine curve of nonzero genus or is $t_{(2)}^2 - 4(8n-15)/(n^2(n-1)^2)$, which is impossible by assumption. 

Using Magma, we checked that if $|V_1\cap L_6|>2$, then one of the following happens: 
\begin{enumerate}[label=(V)-(\roman*), align=left]
\item The Gr\"obner basis we get from $V_1\cap L_{\leq 6}$ has a polynomial in one variable $t_{(3)}$ over $\mathbb{Q}(n)$ such that each of its irreducible factors is either $t_{(3)}$ or $t_{(3)}+1/(2(n-2))$ or can be viewed as a polynomial over $\mathbb{Z}$ in two variables $n(n-1)(n-2)t_{(3)}/3$ and $n$ which defines an affine curve with nonzero genus.
\item $V_1\cap L_{\leq 6}\subseteq L_1\cup L_3\cup L_5\cup\{(4,2),(2^4),(6,2), (4^2),(4,3,2),(4,2^3),(3,2^4),(2^6)\}$
\item $V_1\cap L_{\leq 6}\subseteq L_1\cup L_3\cup L_5\cup\{(4,2),(5,2^2),(4^2),(4,3,2),(4,2^3)\}$
\item $V_1\cap L_{\leq 6}= L_1\cup L_3\cup L_5\cup\{(5),(4^2),(4,3,2),(4,2^3)\}$
\item $V_1\cap L_{\leq 6}= L_1\cup L_3\cup L_5\cup\{(3^2),(6,2),(3^3),(3^2,2^2)\}$ 
\item $V_1\cap L_{\leq 6}= L_1\cup L_3\cup L_5\cup \{(3,2^2),(4^2),(4,3,2),(4,2^3)\}$.
\end{enumerate}
The case (i) can be excluded by \autoref{2-3-22} and Siegel's theorem. Gr\"obner basis computation for the case (iv) forces $\chi((4,2))=0$, so \autoref{2-5-42-odd} excludes this case. For case (vi), Gr\"obner basis computation forces either $\chi((2^2))=0$ or $\rho_\chi((2^2)) = -(2n^2 + 46n - 240)/(n(n-1)(n-2)(n-3))$. The former is already excluded in (I). The latter forces $\omega_\chi((2^2)) = (n^2 + 23n - 120)/4\in \mathbb{Z}$, which implies $n\equiv 0,1$ mod $4$. This contradicts our assumption, so we can also exclude (vi).

For the case (iii), if $V_1 \supset L_1\cup L_3\cup L_5\cup\{(4,2)\}$ and $V_1\cap \{(4^2),(4,3,2),(4,2^3)\}\neq\emptyset$, then $V_1\supseteq L_1\cup L_3\cup L_5\cup\{(4,2),(4^2),(4,3,2),(4,2^3)\}$. Hence, if $|V_1\cap L_6|>2$, then $V_1\cap L_{\leq 6}$ is either $L_1\cup L_3\cup L_5\cup\{(4,2),(4^2),(4,3,2),(4,2^3)\}$ or $L_1\cup L_3\cup L_5\cup\{(4,2),(5,2^2),(4^2),(4,3,2),(4,2^3)\}$. Suppose that $V_1\cap L_{\leq 6} = L_1\cup L_3\cup L_5\cup\{(4,2),(5,2^2),(4^2),(4,3,2),(4,2^3)\}$. Then either $\chi((2^2))=0$ or $\rho_\chi((3)) = (21 (n - 5) (n - 16))/(n (n - 1) (n - 2) (n - 25))$. The former is again excluded in (I), and the latter forces $\omega_\chi((3)) = 7(n-5)(n-16)/(n-25) \in \mathbb{Z}$. Hence $n-25$ cannot have any prime divisor other than $2, 3, 5, 7$, and cannot be divisible by $2^3, 3^3, 5^2, 7^2$, so $n-25 \leq 2^2\cdot 3^2 \cdot 5\cdot 7 = 1260$. Therefore we can also exclude this case, so $V_1\cap L_{\leq 6} = L_1\cup L_3\cup L_5\cup\{(4,2),(4^2),(4,3,2),(4,2^3)\}$, which can be viewed as a part of the case (ii). Therefore, the only possible cases are (V)-(ii) and (V)-(v).

Suppose that $i\in\{2,3\}$, $|V_i\cap L_4|=2$ and $|V_i\cap L_6|\geq 4$. Then we have the following possibilities (after removing those directly violating the assumptions on $n$):
\begin{enumerate}[label=(V)-(\roman*), align=left, resume]
\item The Gr\"obner basis has a polynomial in one variable $t_{(2)}$ over $\mathbb{Q}(n)$ whose irreducible factors either define affine curves of nonzero genus or is one of the following: $t_{(2)}^2 + (-2n^2 - 94n + 864)/(n^2(n-1)^2), t_{(2)}^2 + (-128n + 1008)/(n^2(n-1)^2), t_{(2)}^2 + (-224n + 1284)/(n^2(n-1)^2)$.
\item $i=2$, $V_2\cap (L_2\cup L_4\cup L_6)\subseteq \{(3),(3^2),(3,2^2), (6,2),(5,3),(4,3,2),(3^3),(3^2,2^2),(3,2^4)\}$.
\item $i=3$, $V_3\cap (L_2\cup L_4\cup L_6)\subseteq \{(2^2), (4,2),(3,2^2), (5,2^2),(4^2),(4,3,2),(4,2^3),(3^2,2^2)\}$.
\end{enumerate}

Suppose that $V_1$ is as in (V)-(v). Then exactly two elements of $L_4\setminus V_1 = \{(5),(4,2),(3,2^2),(2^4)\}$ are in $V_2$ and the other two are in $V_3$. Also, at least one of $V_2$ and $V_3$ must also contain at least $4$ elements of $L_6\setminus V_1 = \{(7),(5,3),(5,2^2),(4^2),(4,3,2),(4,2^3),(3,2^4),(2^6)\}$, and the other must contain the remaining elements of $L_6\setminus V_1$. Therefore at least one of $V_2$, $V_3$ must satisfy (V)-(vii), which can be written as $\omega_\chi((2))^2 = n(n+47)/2 -216$ or $4(8n-63)$ or $56n-321$. Note that $56n-321\equiv 3$ mod $4$, so this case is impossible. We checked using Magma that the only cases with $\omega_\chi((2))^2 = n(n+47)/2-216$ are the cases where $i=3$ and $V_3\cap L_{\leq 6}\subseteq \{(2^2),(4),(4,2), (3,2^2),(5,2^2),(4^2),(4,3,2),(4,2^3),(3^2,2^2),(3,2^4)\}$. Also, the cases with $\omega_\chi((2))^2 = 4(8n-63)$ are the cases where $i=3$ and $\{(2^2),(4),(4,2),(3,2^2),(2^6)\}\subseteq V_3\cap L_{\leq 6}\subseteq \{(2^2),(4),(4,2), (3,2^2),(5,2^2),(4^2),(4,3,2),(4,2^3),(3^2,2^2),(2^6)\}$. Both of these cases force $V_2\supseteq \{(3), (5),(2^4), (7),(5,3)\}$. However the Gr\"obner basis for such $V_2$ always contains $t_{(2)}^4 + (-12n^2 + 188n + 384)/(n^2(n-1)^2)t_{(2)}^2 + (12n^2 + 4n - 1488)/(n^3(n-1)^3)$, which viewed as a polynomial in two variables $n(n-1)t_{(2)}/2$ and $n$ defines an affine curve of genus $3$. By Siegel's theorem, this is impossible. 

Therefore the only possibility is that $V_1$ is as in (V)-(ii). If $|V_1\cap L_4|=1$ and $|V_1\cap L_6|\leq 4$, then either $V_2$ or $V_3$ must satisfy one of (V)-(vii),(viii),(ix). Since (V)-(ii), (V)-(vii) and (V)-(ix) do not contain two classes $(5),(3^2)$ which cannot be covered by $V_2$ at the same time by \autoref{4-trans-forbidden}, we cannot have these cases. Hence $V_2$ is as in (V)-(viii), and $V_3\supseteq \{(2^2), (5),(7), (5,2^2)\}\cup (\{(4,2), (2^4),(4^2),(4,2^3),(2^6)\}\setminus V_1)$. Now Gr\"obner basis for $V_3$ and Siegel's theorem forces $\{(4,2),(4^2),(4,2^3),(2^6)\}\subset V_1$, so $V_1\cap L_{\leq 6} = L_1\cup L_3\cup L_5\cup\{(4,2), (4^2),(4,3,2),(4,2^3),(2^6)\}$. By examining $\omega_{\chi}((5))$ in this case, we see that $n\equiv 0$ or $2$ mod $5$, contrary to our assumption. If $|V_1\cap L_6|>4$, or if $|V_1\cap L_4|=2$ and $|V_1\cap L_6|>0$, then Gr\"obner basis computation forces $V_1\cap L_{\leq 6}=L_1\cup L_3\cup L_5\cup\{(4,2),(2^4),(6,2), (4^2),(4,3,2),(4,2^3),(3,2^4),(2^6)\}.$ Again, $\omega_{\chi}((5))$ tells us that $n\equiv 0$ or $3$ mod $5$, contrary to our assumption. Therefore no such triple $(\chi_1,\chi_2,\chi_3)$ exists.
\end{proof}

\begin{rmk}
Note that in the above proof, the condition about $n$ mod $5$ is only used in the last paragraph. We believe that if we continue the process in the same manner for $L_k$'s for $k=7, 8, \dots$, then after few more steps, we would be able remove this modulo $5$ condition, and probably all other conditions except that $n$ is large enough, which would prove \autoref{conjectures}(2). With more steps, we might even be able to do this with more than $3$ characters. However, in order to continue the process in a reasonable amount of time, it seems like we need either better implementation of algorithms than we currently have access to, or additional methods that quickly rule out some possibilities of $V_i$'s.
\end{rmk}

\section{Characters of defect zero}
In \autoref{def0}, we saw that the values of irreducible characters of $S_n$ of $p$-defect zero can be computed from the number $n$, the degree of the character, and the values at the cycles of length less than $p$. In particular, if $\chi$ has $2$-defect zero, or if $\chi$ has $3$-defect zero and is corresponding to a ``self-conjugate'' Young diagram so that $\chi((2))=0$, then the values of $\rho_\chi$ simply become a rational function in one variable $n$. In this section we record some facts about these polynomials. 

It is a well-known fact that the $2$-defect zero characters correspond to the so-called \emph{staircase partitions} $(k,k-1,k-2\dots, 1)$, so that $S_n$ has an irreducible character of $2$-defect zero if and only if $n=k(k+1)/2$ for some $k\in \mathbb{Z}_{>0}$. Hence, we may denote by $\psi_k$ the $2$-defect zero character of $S_{k(k+1)/2}$ that corresponds to the staircase partition $(k,k-1,\dots,1)$, and we will denote by $\rho_k$ the ratio $\rho_{\psi_k}= \psi_k/\psi_k(1)$. (Many authors denote the staircase partition or the corresponding character by $\rho_k$; in this paper, we have been using $\rho$ to mean the ratios instead of characters, so for the sake of consistency, we chose these notations.)

First we fix some notations for convenience.

\begin{dfn}
For positive integers $a, b$, define $$\Delta(a,b) = \prod_{i=1}^{b}\left(a-\frac{i(i+1)}{2}\right) = (a-1)(a-3)(a-6)\cdots (a-b(b+1)/2).$$
Also, for a positive integer $k$, $C(k)$ denotes the $k$th \emph{Catalan number} $$C(k) = \frac{1}{k+1}\binom{2k}{k} = \frac{(2k)!}{(k+1)!k!} .$$
\end{dfn}

\begin{lem}\label{Delta}
For any positive integers $a>b$, $$\Delta\left(\frac{a(a+1)}{2},b\right) = \frac{(a+b+1)!}{(a-b-1)!2^ba(a+1)}$$ where we set $0!=1$.
\end{lem}

\begin{proof}
\begin{align*}
&\Delta\left(\frac{a(a+1)}{2}, b\right) \\=& \left(\frac{a(a+1)}{2}-1\right)\left(\frac{a(a+1)}{2}-3\right)\cdots \left(\frac{a(a+1)}{2}-\frac{b(b+1)}{2} \right) \\=& (a+(a-1)+\cdots + 2)(a+(a-1)+\cdots + 3)\cdots (a+(a-1)+\cdots + (b+1))\\=&\left(\frac{(a-1)(a+2)}{2}\right)\left(\frac{(a-2)(a+3)}{2}\right)\cdots \left(\frac{(a-b)(a+b+1)}{2}\right)\\=& \frac{(a+b+1)!}{(a-b-1)!2^ba(a+1)}.\qedhere
\end{align*}
\end{proof}

\begin{dfn}
For each integer $i>1$, let $Q_{i}(n) = \prod_{a=1}^{i-1}(n-a)$. 
\end{dfn}

We are now ready to state the main result of this section. 
\begin{thm}\label{staircase_poly_existence}
For any partition $\lambda$, there exists a polynomial $P_\lambda$ with the following properties.
\begin{enumerate}[label={\upshape(\alph*)}]
\item $Q_{\supp(\lambda)}(n)\rho_k(\lambda)=P_\lambda(n)$ for any $k$, where $n=k(k+1)/2\geq \supp(\lambda)$.
\item $P_\lambda(n)=0$ for any positive integer $k$ such that either $n=k(k+1)/2<\supp(\lambda)$ or $2k-1<\lambda_1$ where $\lambda_1$ is the largest part of $\lambda$. In other words, $P_\lambda(x)$ is divisible by $\Delta(x, m)$ where $m=\max\{k\in \mathbb{Z}_{>0}\mid k(k+1)/2<\supp(\lambda)\text{ or }2k-1<\lambda_1\}$. 
\item If $\lambda$ has a part of even size, then $P_\lambda=0$.
\end{enumerate}
For the properties below, we assume that $\lambda$ does not have any part of even size.
\begin{enumerate}[resume,label={\upshape(\alph*)}]
\item $\deg P_\lambda = \supp(\lambda)-1-\Vert\lambda\Vert/2$. 
\item The leading coefficient of $P_\lambda$ is $\mathcal{L}(P_\lambda)=(-2)^{-\Vert \lambda\Vert/2}\prod_{i=1}^{r}C((\lambda_i-1)/2)^{a_i}$, where $\lambda=\left(\lambda_1^{(a_1)},\dots,\lambda_r^{(a_r)}\right)$.
\end{enumerate}
\end{thm}

\begin{proof}
We use induction. For the base case we have $P_{(1)}(n)=1$. Now let $\lambda=\left(\lambda_1^{(a_1)},\dots,\lambda_r^{(a_r)}\right)$ be any partition, with odd parts $\lambda_1>\cdots >\lambda_r>1$. Suppose that for each partition $\sigma=\left(\sigma_1^{(b_1)},\dots,\sigma_s^{(b_s)}\right)$ such that $\sigma<\lambda$ in the ordering as in \autoref{part_ordering}, there exists a polynomial $P_{\sigma}$ of degree $\supp(\sigma)-1-\Vert\sigma\Vert/2$ such that: 
\begin{itemize}
\item $Q_{\supp(\sigma)}(n)\rho_k(\sigma)/\rho_k(1)=P_{\sigma}(n)$ for any $k>1$ and $n=k(k+1)/2\geq\supp(\sigma)$,
\item $P_{\sigma}(k(k+1)/2)=0$ for positive integers $k\geq 1$ with either $k(k+1)/2<\supp(\sigma)$ or $2k-1<\lambda_1$, and 
\item $\mathcal{L}(P_{\sigma})=(-2)^{-\Vert\sigma\Vert/2}\prod_{i=1}^{s}C((\sigma_i-1)/2)^{b_i}$.
\end{itemize}
Note that $t_\lambda$ is the largest among the variables $t_L$ with odd $o(L)$ appearing in the polynomial $F_{(2),\overline{\lambda}}$, where $\overline{\lambda}= \left(\lambda_1^{(a_1)},\dots, \lambda_{r-1}^{(a_{r-1})},\lambda_r^{(a_r-1)},\lambda_r-1\right)$. By \autoref{general_poly_rel_thm} applied to $F_{(2),\overline{\lambda}}$ and $\rho_k$ for each $k$ with $n=k(k+1)/2\geq \supp(\overline{\lambda})=\supp(\lambda)-1$, for any $y\in \overline{\lambda}$, we have $$ \frac{n(n-1)}{2} \rho_k((2))\rho_k(\overline{\lambda}) = \rho_k(1)\sum_{x\in (2)}\rho_k(xy).$$ $\rho_k$ vanishes at every element of even order, and $\rho_k(1)>0$, so from the above equality, we get 
\begin{equation}\label{staircase_poly_recurrence}
0 = \sum_{x\in (2)}\rho_k(xy) = \sum_{\substack{\sigma<\lambda\\ o(\sigma)\text{ odd}}}\left\lvert\{x\in (2)\mid xy\in \sigma\}\right\rvert\rho_k(\sigma) + |\{x\in (2)\mid xy\in \lambda\}|\rho_k(\lambda).
\end{equation}

Note that if $\lambda= (\lambda_1^{(a_1)},\dots, \lambda_r^{(a_r)})$, then $$|\{x\in (2) \mid xy\in \lambda\}| = (\lambda_r-1)(n-\supp(\lambda)+1) $$ which is computed as the number of ways to choose one element from the support of the unique cycle of $y$ of length $\lambda_r-1$ and another element not in the support of $y$. 

Let $\sigma\neq \lambda$ be another conjugacy class with odd $o(\sigma)$ which contains $xy$ for some $x\in (2)$. Then it must correspond to a partition of the form $$\alpha_i=(\lambda_1^{(a'_1)},\dots, \lambda_{r-1}^{(a'_{r-1})}, \lambda_r^{(a'_r-1)},\lambda_i+\lambda_r),\text{ where }a'_j = a_j-\delta_{i,j}$$ or $$\beta_j =(\lambda_1^{(a_1)},\dots, \lambda_{r-1}^{(a_{r-1})}, \lambda_r^{(a_r-1)},\lambda_r-j-1, j),\ 1\leq j \leq \frac{\lambda_r-1}{2}, j\text{ odd}.$$ The case $\sigma=\alpha_i$ happens when $x=(ab)$ for some $a$ in the support of the unique cycle of $y$ of length $\lambda_r-1$, and $b$ in the support of some cycle of length $\lambda_i$. Therefore $$\supp(\alpha_i)=\supp(\lambda)-1\text{ and }|\{x\in (2)\mid xy\in \alpha_i\}| = (\lambda_r-1)(a_i-\delta_{i,r})\lambda_i.$$ The case $\sigma=\beta_j$ happens when $x=(ab)$ for some $a,b$ in the support of the unique cycle $\zeta$ of $y$ of length $\lambda_r-1$. More precisely, this happens when $b=\zeta^ja$ or $a=\zeta^j b$. Therefore $$\supp(\beta_j)=\supp(\lambda)-1-\delta_{j,1}-\delta_{3,\lambda_r}\text{ and }|\{x\in (2) \mid xy\in \beta_j\}| = \frac{\lambda_r-1}{1 + \delta_{2j, \lambda_r-1}}.$$ Note that if we use $1\leq j\leq \lambda_r-2$ instead of $i\leq j\leq (\lambda_r-1)/2$, then we get each $\beta_j$ ($1\leq j<(\lambda_r-1)/2$) twice except when $2j = \lambda_r-1$. 

We also know by the induction hypothesis that $\rho_k(\sigma)=P_\sigma(n)/Q_{\supp(\sigma)}(n)$ for some polynomial $P_\sigma$ of degree $\supp(\sigma)-1-\Vert\sigma\Vert/2$ whose leading coefficient is as described above. Now we can rewrite \eqref{staircase_poly_recurrence}, after multiplying by $Q_{\supp(\lambda)-1}(n)$, as
\begin{equation}\label{staircase_P_recurrence}
Q_{s}(n)\rho_k(\lambda) =-\sum_{i=1}^{r}(a_i-\delta_{i,r})\lambda_iP_{\alpha_i}(n)- \sum_{\substack{1\leq j\leq \lambda_r-2\\j \text{ odd}}} \frac{P_{\beta_j}(n)}{2}(n-s+2)^{\delta_{j,1}}(n-s+3)^{\delta_{\lambda_r,3}}
\end{equation}
where $s=\supp(\lambda)$, and the coefficients of $P_{\beta_j}$ are simplified by using the alternative range of $j$ as explained above. The right-hand side is a polynomial in $n$, and we choose this polynomial as our $P_\lambda$. 

Note that by the induction hypothesis, $$P_{\alpha_i}\left(\frac{k(k+1)}{2}\right)=0\text{ and }P_{\beta_j}\left(\frac{k(k+1)}{2}\right)\left(\frac{k(k+1)}{2}-s+2\right)^{\delta_{j,1}}\left(\frac{k(k+1)}{2}-s+3\right)^{\delta_{\lambda_r,3}} = 0$$ for all positive integers $k$ with $k(k+1)/2<s-1$, so $P_\lambda(k(k+1)/2)$ also becomes $0$ for these $k$. If $n=k(k+1)/2 = s-1$, then the equality \eqref{staircase_poly_recurrence} and hence \eqref{staircase_P_recurrence} holds for this $n$ and $Q_s(s-1) =0$, so $P_\lambda(s-1)=0$. Also, since the staircase partition of $k(k+1)/2$ does not have rim hook of length larger than $2k-1$, by the Murnaghan-Nakayama rule we have $P_{\lambda}(k(k+1)/2)=Q_s(k(k+1)/2)\rho_k(\lambda)=0$ if $s\leq k(k+1)/2$ and $2k-1<\lambda_1$.

The partitions $\alpha_i$ and $\beta_j$ satisfies $\Vert\alpha_i\Vert = \Vert\lambda\Vert$ and $\Vert\beta_j\Vert=\Vert\lambda\Vert-1$. By the induction hypothesis, $$\deg P_{\alpha_i} = \supp(\alpha_i)-1-\frac{\Vert\alpha_i\Vert}{2} = \supp(\lambda)-2-\frac{\Vert\lambda\Vert}{2},$$ and $$\deg P_{\beta_j} =\supp(\beta_j)-1 -\frac{\Vert\beta_j\Vert}{2} =\supp(\lambda)-2-\delta_{j,1}-\delta_{\lambda_r,3} -\frac{\Vert\lambda\Vert}{2}+1.$$ We also have $$\mathcal{L}(P_{\alpha_i}) = (-2)^{-\Vert\lambda\Vert/2}C\left(\frac{\lambda_i+\lambda_r-2}{2}\right)\prod_{j=1}^{r}C\left(\frac{\lambda_j-1}{2}\right)^{a_j-\delta_{i,j}-\delta_{r,j}}$$ and $$\mathcal{L}(P_{\beta_j}) = (-2)^{-\Vert\lambda\Vert/2-1}C\left(\frac{j-1}{2}\right)C\left(\frac{\lambda_r-j-2}{2}\right)\prod_{i=1}^{r}C\left(\frac{\lambda_i-1}{2}\right)^{a_i-\delta_{r,i}}. $$ By \eqref{staircase_P_recurrence}, $P_\lambda$ has degree $\supp(\lambda)-1-\Vert\lambda\Vert/2$ and the leading coefficient of $P_{\lambda}$ is the sum of the leading coefficients of $-2^{-1}P_{\beta_j}$, unless this sum is $0$. The sum is:
\begin{align*}
&\sum_{\substack{1\leq j\leq \lambda_r-1\\j\text{ odd}}}-2^{-1}\mathcal{L}(P_{\beta_j})\\=& \sum_{\substack{1\leq j\leq \lambda_r-1\\j\text{ odd}}} -\frac{1}{2} (-2)^{-\Vert\lambda\Vert/2-1}C\left(\frac{j-1}{2}\right)C\left(\frac{\lambda_r-j-2}{2}\right)\prod_{i=1}^{r}C\left(\frac{\lambda_i-1}{2}\right)^{a_i-\delta_{r,i}}\\=&(-2)^{-\Vert\lambda\Vert/2}\left( \prod_{i=1}^{r}C\left(\frac{\lambda_i-1}{2}\right)^{a_i-\delta_{r,i}}\right)\sum_{0\leq h\leq \frac{\lambda_r-j}{2}-1} C(h)C\left(\frac{\lambda_r-1}{2}-1-h\right)\\=& (-2)^{-\Vert\lambda\Vert/2}\left( \prod_{i=1}^{r}C\left(\frac{\lambda_i-1}{2}\right)^{a_i-\delta_{r,i}}\right)C\left(\frac{\lambda_r-1}{2}\right)=(-2)^{-\Vert\lambda\Vert/2}\prod_{i=1}^{r}C\left(\frac{\lambda_i-1}{2}\right)^{a_i}
\end{align*}
where the penultimate equality follows from the recurrence relation of Catalan numbers. Since this is nonzero, this is $\mathcal{L}(P_\lambda)$, and $\deg P_\lambda=\supp(\lambda)-1-\Vert\lambda\Vert/2$.
\end{proof}

\begin{crl}\label{staircase_cycle_thm}
For any positive integer $k$ and $n=k(k+1)/2$,
\begin{equation}\label{staircase_cycle}\rho_k((2r+1))=\frac{P_{(2r+1)}(n)}{Q_{2r+1}(n)}=(-2)^{-r}C(r)\frac{\Delta(n,r)}{Q_{2r+1}(n)}
\end{equation} for every positive integer $r$ such that $2r+1\leq n$. Also,
\begin{align*}&\rho_k((2r+1,3)) = \left(n^2 - \frac{12r^2+19r+8}{r+2}n + (2r+1)(2r+3)(r+1)\right)\frac{C(r)C(1)\Delta(n, r)}{(-2)^{r+1}Q_{2r+4}(n)} \\&\rho_k((2r+1, 5)) = \left(n^3 - \frac{30r^2+49r+27}{r+3}n^2 + \frac{20r^4+120r^3+265r^2+228r+69}{r+3}n\right. \\&\left.\ \ \ \ \ \ \ \ \ \ \ \ \ \ \ \ \ \ \ \ -\frac{1}{2}(2r+1)(r+1)(2r+3)(r+2)(2r+5) \right)\frac{C(r)C(2)\Delta(n,r)}{(-2)^{r+2}Q_{2r+6}(n)}
\end{align*} for every positive integer $r$ such that the permutation is defined in $S_n$.
\end{crl}

\begin{proof}
By \autoref{staircase_poly_existence}, we know that $P_{(2r+1)}(x)$ has leading coefficient $(-2)^{-r}C(r)$, has degree $r$, and is divisible by $\Delta(x,r)$. Therefore we must have \eqref{staircase_cycle}. The other equalities follow from \eqref{staircase_cycle} and repeated applications of \eqref{staircase_P_recurrence}.
\end{proof}

\begin{rmk}
The proof of \autoref{staircase_poly_existence} gives us a recursive algorithm to compute the polynomial $P_\lambda(x)$ from the polynomials $P_\sigma(x)$ for the cycle types $\sigma<\lambda$: we can use the recurrence relation \eqref{staircase_P_recurrence}, with the left-hand side replaced by $P_{\lambda}(x)$ and each occurrence of $n$ on the right-hand side by $x$. 

It would be helpful if we can find a reasonably simple extension of \eqref{staircase_cycle} that gives a closed-form formula for $P_\lambda(x)$ for general $\lambda$, but we do not know if this is possible.
\end{rmk}

\begin{crl}
For each class $\lambda$ of odd order, $$\rho_k(\lambda)\frac{Q_{\supp(\lambda)}(\frac{k(k+1)}{2})}{\left(\frac{k(k+1)}{2}\right)^{\supp(\lambda)-1-\Vert \lambda\Vert/2}(-2)^{-\Vert \lambda\Vert/2}\prod_{i=1}^{r}C((\lambda_i-1)/2)^{a_i}} \rightarrow 1\text{ as }k\rightarrow \infty.$$ Also, there are at most $\supp(\lambda)-1-\Vert\lambda\Vert/2$ positive integers $k$ such that $\psi_k(\lambda)=0$. 
\end{crl}

\begin{proof}
For the first part, the left-hand side is $P_\lambda(n)$ divided by its leading term, where $n=k(k+1)/2$. For the second part, $\psi_k(\lambda)=0$ if and only if $P_\lambda(k(k+1)/2)=0$, and since $P_\lambda$ is a polynomial, the number of $k$'s satisfying this cannot exceed $\deg P_\lambda = \supp(\lambda)-1-\Vert\lambda\Vert/2$. 
\end{proof}

The irreducible characters having $3$-defect $0$ which also vanishes at $(2)$ correspond to the partitions of the form $$(3k - 2, 3k - 4, . . . , k + 4, k + 2, k, (k - 1)^2, (k - 2)^2, \dots , 2^2, 1^2)\vdash k(3k-2)$$ or $$(3k, 3k - 2, . . . , k + 4, k+2, k^2, (k - 1)^2,\dots , 2^2, 1^2)\vdash k(3k+2).$$ So $S_n$ has a (unique) such character if and only if $n = k(3k-2)$ or $n=k(3k+2)$ for some positive integer $k$. Such numbers are called the \emph{generalized octagonal numbers}, listed on the OEIS\cite{OEIS} as A001082. Let us denote by $\tau_n$ this character in $S_n$. We have the following analogue of \autoref{staircase_poly_existence}.

\begin{thm}\label{3-def-zero}
For any partition $\lambda$, there exists a polynomial $\tilde{P}_\lambda$ with the following properties.
\begin{enumerate}[label={\upshape(\alph*)}]
\item $Q_{\supp(\lambda)}(n)\rho_{\tau_n}(\lambda) = \tilde{P}_\lambda(n)$ if $n\geq \supp(\lambda)$.
\item $\tilde{P}_\lambda(k(3k+2))=0$ if either $6k-1<\lambda_1$ or $k(3k+2)<\supp(\lambda)$. Also, $\tilde{P}_\lambda(k(3k-2))=0$ if either $6k-5<\lambda_1$ or $k(3k-2)<\supp(\lambda)$. In other words, $\tilde{P}_\lambda(x)$ is divisible by $(x-n)$ for each of these $n=k(3k\pm 2)$.
\item If $\lambda$ has a part of size divisible by $3$ or if $\Vert\lambda\Vert$ is odd, then $\tilde{P}_\lambda=0$. 
\item If $\tilde{P}_\lambda \neq 0$, then $\deg \tilde{P}_\lambda \leq \Vert\lambda\Vert$.
\end{enumerate}
\end{thm}

\begin{proof}
We argue as in the proof of \autoref{staircase_poly_existence}. So assume that the statement holds for all partitions less than $\lambda=(\lambda_1^{a_1},\dots,\lambda_r^{a_r})$ in the ordering as in \autoref{part_ordering}. Also assume that $3\nmid o(\lambda)$ and $2\mid \Vert \lambda\Vert$, so that it has no part of size divisible by $3$, and the number of parts of even size is even.

If $\lambda$ has a part of size $2$, then $t_{\lambda}$ is the largest variable appearing in $F_{(2),\overline{\lambda}}$ where $\overline{\lambda}$ is the partition obtained by removing one part of size $2$ from $\lambda$. Therefore, by \autoref{general_poly_rel_thm}, for any $y\in \overline{\lambda}$ and any $n\geq \supp(\overline{\lambda}) = \supp(\lambda)-2$ of the form $n=k(3k\pm 2)$ for some $k\in \mathbb{Z}_{>0}$,
\begin{align*}
0 =& \sum_{x\in (2)}\rho_{\tau_n}(xy) = \sum_{\substack{\sigma<\lambda\\ 3\nmid o(\sigma)}}\left\lvert\{x\in (2)\mid xy\in \sigma\}\right\rvert\rho_{\tau_n}(\sigma) + |\{x\in (2)\mid xy\in \lambda\}|\rho_{\tau_n}(\lambda).
\end{align*}
Note that $|\{x\in (2)\mid xy\in \lambda\}| = (n-\supp(\lambda)+2)(n-\supp(\lambda)+1)/2$. Also, suppose that $\sigma<\lambda$ can be obtained as the class of $xy$ for some $x\in (2)$. Then $\supp(\lambda)-\supp(\sigma)\in \{1,2,3,4\}$. It is $1$ if $\sigma$ is obtained by increasing the size of a part $\lambda_i$ of $\overline{\lambda}$; in this case, $|\{x\in (2)\mid xy\in \sigma\}|=(a_i-\delta_{i,r})\lambda_i(n-\supp(\lambda)+2)$. $\supp(\lambda)-\supp(\sigma)=2$ if $\sigma$ is obtained by merging two parts of $\overline{\lambda}$ into one or splitting one part of $\overline{\lambda}$ into two parts of sizes at least $2$; in this case $|\{x\in (2)\mid xy\in \sigma\}|=$ 
\begin{itemize}
\item $a_i(a_j-\delta_{j,r})\lambda_i\lambda_j$ if we merged two parts of sizes $\lambda_i>\lambda_j$, 
\item $\frac{(a_i-\delta_{i,r})(a_i-1-\delta_{i,r})}{2}\lambda_i^2$ if we merged two parts of same sizes $\lambda_i$, 
\item $a_i\lambda_i$ if we split a part of size $\lambda_i$ into two parts of different sizes, and
\item $\frac{a_i}{2}\lambda_i$ if we split a part of size $\lambda_i$ into two parts of different sizes.
\end{itemize}
$\supp(\lambda)-\supp(\sigma)=3$ if $\sigma$ is obtained by reducing the size of a part of $\overline{\lambda}$ of size larger than $2$ by one; in this case $|\{x\in (2)\mid xy\in \sigma\}|=a_i\lambda_i$. Finally, $\supp(\lambda)-\supp(\sigma)=4$ if $\sigma$ is obtained by removing a part of $\overline{\lambda}$ of size $2$; in this case $|\{x\in (2)\mid xy\in \sigma\}|=a_r-1$. Therefore, 
{\allowdisplaybreaks \begin{align*}
&Q_{\supp(\lambda)}(n)\rho_{\tau_n}(\lambda) \\=&-2Q_{\supp(\lambda)-2}(n)\sum_{\substack{\sigma<\lambda\\ 3\nmid o(\sigma)}}\left\lvert\{x\in (2)\mid xy\in \sigma\}\right\rvert\rho_{\tau_n}(\sigma)\\=&-2\sum_{\substack{\sigma<\lambda\\ 3\nmid o(\sigma)\\\supp(\sigma)=\supp(\lambda)-1}}(\text{positive integer})(n-\supp(\lambda)+2)Q_{\supp(\lambda)-2}(n)\rho_{\tau_n}(\sigma) \\& -2\sum_{\substack{\sigma<\lambda\\ 3\nmid o(\sigma)\\\supp(\sigma)\leq \supp(\lambda)-2}}(\text{positive integer})Q_{\supp(\lambda)-2}(n)\rho_{\tau_n}(\sigma) \\=&-2\sum_{\substack{\sigma<\lambda\\ 3\nmid o(\sigma)\\\supp(\sigma)=\supp(\lambda)-1}}(\text{positive integer})\tilde{P}_\sigma(n) \\& -2\sum_{\substack{\sigma<\lambda\\ 3\nmid o(\sigma)\\\supp(\sigma)\leq \supp(\lambda)-2}}(\text{positive integer})\left(\prod_{i=3}^{\supp(\lambda)-\supp(\sigma)}(n-\supp(\lambda)+i)\right)\tilde{P}_\sigma(n).
\end{align*}}
We can take the right-hand side with $n$ replaced by $x$ as the polynomial $\tilde{P}_{\lambda}(x)$. 

By the induction hypothesis applied to each term on the right-hand side, $\tilde{P}_\lambda(n)=0$ if $n=k(3k\pm 2)$ for some $k\in \mathbb{Z}_{>0}$ and $n<\supp(\lambda)-2$. Also, if $n=k(3k\pm 2)$ equals $\supp(\lambda)-1$ or $\supp(\lambda)-2$, the above equality still holds and the left-hand side is $0$, so $\tilde{P}_\lambda(n)=0$ for these $n$. Also, $$\deg\tilde{P}_\lambda\leq \max\{\deg(\tilde{P}_\sigma) + \max(0, \supp(\lambda)-\supp(\sigma)-2) \mid \sigma<\lambda, \supp(\sigma)<\supp(\lambda)\}.$$ If $xy\in \sigma$ for some $x\in(2)$ and $\Vert\sigma\Vert = \Vert\lambda\Vert$, then $\supp(\sigma)\geq \supp(\lambda)-1$. If we instead have $\Vert\sigma\Vert<\Vert\lambda\Vert$, then $\Vert\sigma\Vert = \Vert\lambda\Vert-2$. Therefore, by the induction hypothesis, we get $$\deg\tilde{P}_\lambda\leq \max\{\Vert\sigma\Vert + \Vert\lambda\Vert - \Vert\sigma\Vert\mid \sigma<\lambda, \supp(\sigma)<\supp(\lambda)\} = \Vert\lambda\Vert.$$

If $\lambda$ has no part of size $2$, then define $\overline{\lambda}$ as the partition obtained by choosing a part of $\lambda$ of smallest size larger than $1$ and reducing its size by $2$. Then $\lambda$ is the largest partition without any part of size divisible by $3$ among the partitions whose corresponding variable appears in $F_{(3), \overline{\lambda}}$. By \autoref{general_poly_rel_thm}, for any $y\in\overline{\lambda}$,
\begin{equation*}
0 = \sum_{x\in (3)}\rho_{\tau_n}(xy) = \sum_{\substack{\sigma<\lambda\\ 3\nmid o(\sigma)}}\left\lvert\{x\in (3)\mid xy\in \sigma\}\right\rvert\rho_{\tau_n}(\sigma) + |\{x\in (3)\mid xy\in \lambda\}|\rho_{\tau_n}(\lambda)
\end{equation*}
As above, for any $n$ of the form $k(3k\pm 2)$ with $n\geq \supp(\lambda)-3$, we have the equality
\begin{align*}
&Q_{\supp(\lambda)}(n)\rho_{\tau_n(\lambda)} \\=& -3\sum_{\substack{\sigma<\lambda\\3\nmid o(\sigma)\\\supp(\sigma)\geq \supp(\lambda)-2}}(\text{positive integer})\tilde{P}_\sigma(n) \\=& -3\sum_{\substack{\sigma<\lambda\\3\nmid o(\sigma)\\\supp(\sigma)= \supp(\lambda)-3}}(\text{positive integer})(n-\supp(\lambda)+3)\tilde{P}_\sigma(n) \\&-3\sum_{\substack{\sigma<\lambda\\3\nmid o(\sigma)\\3\leq\supp(\lambda)-\supp(\sigma)\leq 5}}(\text{positive integer})\left(\prod_{i=4}^{\supp(\lambda)-\supp(\sigma)}(n-\supp(\lambda)+i)\right)\tilde{P}_\sigma(n).
\end{align*}
We may choose the right-hand side with $n$ replaced by $x$ as the polynomial $\tilde{P}_\lambda(x)$. By the induction hypothesis applied to each term on the right-hand side, we get $\tilde{P}_{\lambda}(n)=0$ if $n=k(3k\pm 2)$ for some $k\in \mathbb{Z}_{>0}$ and $n<\supp(\lambda)-3$. If $n=k(3k\pm 2)$ and $\supp(\lambda)-3\leq n\leq \supp(\lambda)-1$, then the above equality holds with the left-hand side being $0$, so $\tilde{P}_\lambda(n)=0$ for these $n$. Also, $$\deg\tilde{P}_\lambda\leq \max\{\deg\tilde{P}_\sigma + \max(0, \supp(\lambda)-\supp(\sigma)-3, \delta_{3, \supp(\lambda)-\supp(\sigma)}) \mid \sigma<\lambda, \supp(\sigma)\leq\supp(\lambda)\}.$$ Again, we can easily check that $\max(0, \supp(\lambda)-\supp(\sigma)-3, \delta_{3, \supp(\lambda)-\supp(\sigma)})\leq \Vert\lambda\Vert-\Vert\sigma\Vert$. By the induction hypothesis we get $\deg\tilde{P}_\lambda\leq \Vert\lambda\Vert$. Therefore, regardless of whether $\lambda$ has a part of size $2$ or not, we have (a)-(d).
\end{proof}

A deeper analysis of the above situations might yield an exact formula for the degree and leading coefficient of $\tilde{P}_\lambda$, and a reasonably simple expression of $\tilde{P}_{(r)}$ for each cycle $(r)$ similar to \eqref{staircase_cycle}. We decided not to do it here, as the argument is expected to be very lengthy and cumbersome. 

\begin{qst}\label{Final_Questions}
(1) Is there a simple formula for $P_\lambda$ and $\tilde{P}_\lambda$ for general $\lambda$? More generally, is there a simple expression of $T_\lambda$ defined in \autoref{char_from_cycles}?\\
(2) Can we find a formula or at least an upper bound for the coefficients of $P_\lambda$, $\tilde{P}_\lambda$ or $T_\lambda$? 
\end{qst}

A possible application of these questions, if the answers are good enough, is the Tensor Square conjecture. Heide, Saxl, Tiep and Zalesski \cite{HSTZ} proved that for every finite simple groups of Lie type except $\PSU_n(q)$ with $n$ coprime to $2(q+1)$, every complex irreducible character appears as an irreducible constitute of the tensor square of the Steinberg character. Based on this result, they conjectured that the alternating groups also have a complex irreducible character with the same property, namely its tensor square has every irreducible character as an irreducible constitute. One of the authors, Jan Saxl, conjectured that the staircase (2-defect zero) characters $\psi_k$ have this property; this special case is sometimes called the Saxl conjecture. 

Since the multiplicity of an irreducible constitute can be computed using the inner product of characters, if we have a good estimate of character values, we might be able to use it to bound the inner product $[\chi,\psi_k^2]$, or equivalently $[\rho_\chi,\rho_k^2]$, away from zero. As the values of $\rho_k^2$ are positive, it would be nice if we can find a good lower bound of $\{\rho_\chi(\lambda)\mid \chi\in\Irr(S_n)\}$ for each $\lambda$ (of odd order), in terms of $n$. The following are some easy examples of lower bounds, which do not seem to be good enough for our purpose.

\begin{prop}
For any $n$ and $\chi\in\Irr(S_n)$, $\rho_\chi((2^2))> -(4n-6)/(n-2)(n-3)$ and $\rho_\chi((3^2)) >\frac{36 n^5 - 324 n^4 + 1200 n^3 - 1791 n^2 - 120 n + 3600}{4n(n-1)(n-2)(n-3)(n-4)(n-5)} $.
\end{prop}
\begin{proof}
Recall that
$$T_{(2^2)} = t_{(2^2)} + \frac{4t_{(3)}}{n - 3} + \frac{-n(n-1)t_{(2)}^2}{(n-2)(n-3)} + \frac{2}{(n-2)(n-3)}.$$
Since $\rho_\chi(\lambda)\leq 1$ for all $\lambda$, and $\rho_\chi((3))=1$ if and only if $\rho_\chi((2^2))=1$, we get $$\rho_\chi((2^2)) = - \frac{4\rho_\chi((3))}{n - 3} - \frac{-n(n-1)\rho_\chi((2))^2}{(n-2)(n-3)} - \frac{2}{(n-2)(n-3)}>- \frac{4(n-2)+2}{(n-2)(n-3)}.$$ Similarly, 
\begin{align*}
T_{(3^2)} :=& t_{(3^2)} + \frac{9t_{(5)}}{n - 5} + \frac{-n(n-1)(n-2)t_{(3)}^2}{(n-3)(n-4)(n-5)} + \frac{(9n - 60)t_{(3)} }{(n-3)(n-4)(n-5)}\\&+ \frac{9n(n-1)t_{(2)}^2}{(n-2)(n-3)(n-4)(n-5)} + \frac{3(n - 8)}{(n-2)(n-3)(n-4)(n-5)}
\end{align*}
so we get 
\begin{align*}
\rho_\chi((3^2)) &>- \frac{9}{n-5} + \frac{n(n-1)(n-2)t_{(3)}^2 - (9n-60)t_{(3)}}{(n-3)(n-4)(n-5)} - \frac{9n+12}{(n-3)(n-4)(n-5)} \\&\geq -\frac{9n^2- 54n + 120}{(n-3)(n-4)(n-5)} -\frac{(9n-60)^2}{4n(n-1)(n-2)(n-3)(n-4)(n-5)}\\&= -\frac{36 n^5 - 324 n^4 + 1200 n^3 - 1791 n^2 - 120 n + 3600}{4n(n-1)(n-2)(n-3)(n-4)(n-5)} .\qedhere
\end{align*}
\end{proof}

\section{Realizing polynomials as values of characters}
It is natural to ask if there is any (polynomial) relation, other than \eqref{symm_poly_rel} and their variants including those given in \autoref{char_from_cycles}, between the values of complex irreducible characters of $S_n$ that works for all $n$ and all $\chi\in\Irr(S_n)$. The following computations suggest that the answer might be negative.

\begin{prop}
Let $k, a_1,\dots,a_r,b_1,\dots,b_r,c_1\dots,c_r\in \mathbb{Z}_{\geq 0}$, $a_1>a_2>\dots>a_r>0$, and suppose that $k$ is sufficiently large. Let $n=\left(\sum_{i=1}^{r}2a_ik+b_i+c_i\right)-r^2$ and $$\lambda=(a_1k+b_1, a_2k+b_2,\dots,a_rk+b_r,r^{a_rk+c_r-r},r-1^{(a_{r-1}-a_r)k+c_{r-1}-c_r},\dots,1^{(a_1-a_2)k+c_1-c_2})$$be a partition of $n$ (so the first $r$ columns have sizes $a_ik+c_i$). Let $A = \sum_{i=1}^{r}a_i$, $B= \sum_{i=1}^{r}a_i(b_i-c_i)$, $C = \sum_{i=1}^{r}(-1)^{i+1}(b_i-c_i)$, $D=\sum_{i=1}^{r}(2i-1)(b_i-c_i)$, $E=\sum_{i=1}^{r}b_i^2-c_i^2$, and $F=\left(\sum_{i=1}^{r}b_i+c_i\right)-r^2 = n-2Ak$. Then 
\begin{align*}\frac{n(n-1)}{2}\frac{\chi_{\lambda}((2))}{\chi_{\lambda}(1)} = \frac{B}{2A}n + \frac{E-D}{2} - \frac{BF}{2A}.
\end{align*}
\end{prop}

\begin{proof}
We prove this for the case $r=2$; the general case can be proved similarly. In the Young diagram of shape $\lambda$, there are exactly $4$ rim hooks of length $2$, at the end of the first row, the second row, the first column, and the second column. Let $\lambda_{r1}$, $\lambda_{r2}$, $\lambda_{c1}$, and $\lambda_{c2}$ be the partitions obtained by removing each of these rim hooks of length $2$. By Murnaghan-Nakayama rule, we get: 
\begin{align*}
\chi_\lambda((2)) =& \chi_{\lambda_{r1}}(1) + \chi_{\lambda_{r2}}(1) - \chi_{\lambda_{c1}}(1) - \chi_{\lambda_{c2}}(1)
\end{align*}
The hook length formula tells us that if $\mu$ is a partition of the form $(\mu_1, \mu_2, 2^{\mu_3-2},1^{\mu_4-\mu_3})$, then 
\begin{align*}
\chi_\mu(1) = \frac{(\mu_1+\mu_2+\mu_3+\mu_4-4)!}{\frac{(\mu_1-1)!}{\mu_1-\mu_2+1} \frac{(\mu_4-1)!}{\mu_4-\mu_3+1}(\mu_2-2)!(\mu_3-2)!(\mu_2+\mu_3-3)(\mu_1+\mu_3-2)(\mu_2+\mu_4-2)(\mu_1+\mu_4-1) }
\end{align*}
Hence we get
\begin{align*}
\frac{\chi_\lambda((2))}{\chi_\lambda(1)} =& \frac{\chi_{\lambda_{r1}}(1)}{\chi_\lambda(1)} + \frac{\chi_{\lambda_{r2}}(1)}{\chi_\lambda(1)} - \frac{\chi_{\lambda_{c1}}(1)}{\chi_\lambda(1)} - \frac{\chi_{\lambda_{c2}}(1)}{\chi_\lambda(1)}\\=& \frac{(a_1k+b_1-1)(a_1k+b_1-2)(a_1k+b_1+a_2k+c_2-2)(a_1k+b_1+a_1k+c_1-1)}{n(n-1)\frac{a_1k+b_1-a_2k-b_1+1}{a_1k+b_1-a_2k-b_2-1}(a_1k+b_1+a_2k+c_2-4)(a_1k+b_1+a_1k+c_1-3)} \\&+\frac{(a_2k+b_2-2)(a_2k+b_2-3)(a_2k+b_2+a_2k+c_2-3)(a_2k+b_2+a_1k+c_1-2)}{n(n-1)\frac{a_1k+b_1-a_2k-b_2+1}{a_1k+b_1-a_2k-b_2+3} (a_2k+b_2+a_2k+c_2-5)(a_2k+b_2+a_1k+c_1-4)} \\&-\frac{(a_2k+c_2-2)(a_2k+c_2-3)(a_2k+b_2+a_2k+c_2-3)(a_1k+b_1+a_2k+c_2-2)}{n(n-1)\frac{a_1k+c_1-a_2k-c_2+1}{a_1k+c_1-a_2k-c_2+3}(a_2k+b_2+a_2k+c_2-5)(a_1k+b_1+a_2k+c_2-4)} \\&-\frac{(a_1k+c_1-1)(a_1k+c_1-2)(a_1k+c_1+a_2k+b_2-2)(a_1k+b_1+a_1k+c_1-1)}{n(n-1)\frac{a_1k+c_1-a_2k-c_2+1}{a_1k+c_1-a_2k-c_2-1}(a_2k+b_2+a_1k+c_1-4)(a_1k+b_2+a_1k+c_1-3)} \\=&\frac{2k(a_1(b_1 -c_1) + a_2(b_2 - c_2)) + b_1^2 - b_1 + b_2^2 - 3b_2 - c_2^2 + 3c_2 - c_1^2 + c_1}{n(n-1)} \\=& \frac{2B\frac{n-F}{2A} + E-D}{n(n-1)} = \frac{2}{n(n-1)} \left(\frac{B}{2A}n + \frac{E-D}{2} - \frac{BF}{	2A}\right) .
\end{align*}
\end{proof}

Similarly we have the following.

\begin{prop}\label{two_hooks}
Let $a,b,c,d,e\in \mathbb{Z}_{\geq 0}$, and suppose that $a$ is sufficiently large. Let $n = 2a+b+c+d+e$, and let $\lambda= (a+b, 2+c, 2^d, 1^{a-d-2+e})$ be a partition of $n$. Then $$\frac{n(n-1)}{2}\frac{\chi_\lambda((2))}{\chi_\lambda(1)} = \frac{b-e}{2}n + \frac{(c+d+1)(-b+c-d+e)}{2}.$$
\end{prop}

In particular, by choosing $r=2$, $b=e$ and $d=c-1$, we can get $\frac{n(n-1)}{2}\frac{\chi_\lambda((2))}{\chi_\lambda(1)} = c$, and by choosing $c=d-1$ instead, we can get $\frac{n(n-1)}{2}\frac{\chi_\lambda((2))}{\chi_\lambda(1)} = -d.$ So there is no restriction on the number $\frac{n(n-1)}{2}\frac{\chi_\lambda((2))}{\chi_\lambda(1)}$.

We can also realize many polynomials of degree $2$:

\begin{prop}
Let $k, a_1,\dots,a_r, b_1,\dots,b_r\in \mathbb{Z}_{\geq 0}$, $a_1> \cdots > a_r > 0$, and $k$ is large enough compared to $b_1,\dots,b_r$. Let $n = (a_1+\cdots + a_r)k + b_1+\cdots + b_r$, and let $\lambda = (a_1k+b_1,a_2k+b_2, \dots, a_rk+b_r)$ be a partition of $n$. Let $$A_1 = \sum_{i=1}^{r}a_i,\ B=\sum_{i=1}^{r}b_i,\ C= \sum_{i=1}^{r}a_i^2,\ D = \sum_{i=1}^{r}a_ib_i,\ E=\sum_{i=1}^{r}(2i-1)a_i,\ F=\sum_{i=1}^{r}(2i-1)b_i, G=\sum_{i=1}^{r}b_i^2.$$ Then 
\begin{align*}
\frac{n(n-1)}{2}\frac{\chi_\lambda((2))}{\chi_\lambda(1)} =& \frac{C}{2A^2}n^2 + \left( \frac{2D - E}{2A} - \frac{BC}{A^2}\right)n + \frac{CB^2 - A\left(2D - E\right)B + A^2\left(G - F\right)}{2A^2}
\end{align*}
\end{prop}

\begin{proof}
We prove this for the case $r=2$; the general case can be proved in a similar way. By Murnaghan-Nakayama rule and the hook length formula,
\begin{align*}
\frac{\chi_\lambda((2))}{\chi_\lambda(1)} =&  \frac{\chi_{(ak+b-2,ck+d)}(1) +\chi_{(ak+b,ck+d-2)}(1)}{\chi_\lambda(1)}\\=& \frac{(ak+b-ck-d)(ak+b-ck-d-1)(ak+b+1)(ak+b)}{(ak+b-ck-d)(ak+b-ck-d+1)n(n-1)}\\&+\frac{(ak+b-ck-d+2)(ak+b-ck-d+3)(ck+d-1)(ck+d)}{n(n-1)(ak+b-ck-d+1)(ak+b-ck-d+2)}\\=& \frac{(ak+b-ck-d-1)(ak+b+1)(ak+b) + (ak+b-ck-d+3)(ck+d-1)(ck+d)}{n(n-1)(ak+b-ck-d+1)}\\=& \frac{(a^2+c^2)k^2 +(2ab-a+2cd-3c)k + b^2-b+d^2-3d }{n(n-1)} \\=&\frac{\frac{a^2+c^2}{(a+c)^2}(n-b-d)^2 +\frac{2ab-a+2cd-3c}{a+c}(n-b-d) + b^2-b+d^2-3d }{n(n-1)}\\=& \frac{a^2+c^2}{n(n-1)(a+c)^2}n^2 + \left( \frac{-2(a^2+c^2)(b+d)}{n(n-1)(a+c)^2} + \frac{2ab+2cd-a-3c}{n(n-1)(a+c)}\right)n \\&+ \frac{(a^2+c^2)(b+d)^2 - (a+c)(2ab-a+2cd-3c)(b+d) + (a+c)^2(b^2-b+d^2-3d)}{n(n-1)(a+c)^2}\\=& \frac{C}{n(n-1)A^2}n^2 + \left( \frac{2D - E}{n(n-1)A} - \frac{2BC}{n(n-1)A^2}\right)n + \frac{CB^2 - A\left(2D - E\right)B + A^2\left(G - F\right)}{n(n-1)A^2}.
\end{align*}

\end{proof}

\begin{prop}\label{two_hooks_3}
In the situation of \autoref{two_hooks}, 
\begin{align*}
\frac{n(n-1)(n-2)}{3}\frac{\chi_\lambda((3))}{\chi_\lambda(1)} =& \frac{n^3}{12} - \frac{c+d+3}{4}n^2 + \left(\frac{(b-e)^2+(c+d+1)^2}{4}+\frac{5}{12}\right)n \\&- \frac{(c+d+1)(b-c+d-e)(b+c-d-e)}{4}. 
\end{align*}  
\end{prop}

This and \autoref{two_hooks} show that there is no general polynomial relation between $n$ and two of $\chi((2)), \chi((3))$ and $\chi((2^2))$:

\begin{crl}
There is no nonzero polynomial in three variables $N, t_{(2)}, t_{(3)}$ over $\mathbb{Q}$ whose zero set includes triples $(N, t_{(2)},t_{(3)}) = (n, \chi((2))/\chi(1), \chi((3))/\chi(1))$ for all $n\in \mathbb{Z}_{\geq 3}$ and all $\chi\in\Irr(S_n)$.
\end{crl}

\begin{proof}
Suppose that $F\in \mathbb{Q}[N,t_{(2)},t_{(3)}]$ is a polynomial whose zero set includes all such triples. We claim that $F=0$. Let $a,b,c,d,e\in \mathbb{Z}_{\geq 0}$ with $a$ being large enough compared to $b,c,d,e$, and let $n$ and $\lambda$ be as in \autoref{two_hooks}. Then $F(n,\chi_\lambda((2))/\chi_\lambda(1), \chi_\lambda((3))/\chi_\lambda(1))=0$. We may multiply $F$ by a power of $N(N-1)(N-2)$ to rewrite it as: $$F(N, t_{(2)},t_{(3)})(N(N-1)(N-2))^s =  G\left(N, \frac{N(N-1)}{2}t_{(2)}, \frac{N(N-1)(N-2)}{3}t_{(3)}\right)$$ for some trivariate polynomial $G$ over $\mathbb{Q}$. By writing $\chi_\lambda((2))/\chi_\lambda(1)$ and $\chi_\lambda((3))/\chi_\lambda(1)$ in terms of $n,b,c,d,e$ using \autoref{two_hooks} and \autoref{two_hooks_3}, we get 
\begin{align*}
G\left(n, \frac{b-e}{2}n\right.& + \frac{(c+d+1)(-b+c-d+e)}{2},  \frac{n^3}{12} - \frac{c+d+3}{4}n^2 \\+& \left(\frac{(b-e)^2+(c+d+1)^2}{4}+\frac{5}{12}\right)n \left.- \frac{(c+d+1)(b-c+d-e)(b+c-d-e)}{4}\right) \\=& F\left(n, \frac{\chi_\lambda((2))}{\chi_\lambda(1)}, \frac{\chi_\lambda((3))}{\chi_\lambda(1)}\right)(n(n-1)(n-2))^s= 0
\end{align*}
for all $n,b,c,d,e$ as above. If we choose $H\in \mathbb{Q}[N,B,C,D,E]$ as
\begin{align*}
H&(N,B,C,D,E) := G\left(N, \frac{B-E}{2}N\right. + \frac{(C+D+1)(-B+C-D+E)}{2},  \frac{N^3}{12} - \frac{C+D+3}{4}N^2 \\+& \left(\frac{(B-E)^2+(C+D+1)^2}{4}+\frac{5}{12}\right)N \left.- \frac{(C+D+1)(B-C+D-E)(B+C-D-E)}{4}\right)
\end{align*}
then $H(n,b,c,d,e)=0$ for all $n,b,c,d,e$ as above, so $H=0$. 

If $G$ is nonzero, let $N^u G_1$ be the sum of the terms of $G$ with the smallest degree $(=u)$ in the first variable, where $G_1$ is a polynomial in the last two variables of $G$. Then the sum of the terms of $H$ with the same $N$-degree must be $$N^uG_1\left(\frac{(C+D+1)(-B+C-D+E)}{2},- \frac{(C+D+1)(B-C+D-E)(B+C-D-E)}{4}\right).$$ This must be zero since $H=0$. Therefore, for any $b,c,d,e\in \mathbb{Z}_{\geq 0}$, we can set $X = \frac{(c+d+1)(-b+c-d+e)}{2}$ and $Y=\frac{b+c-d-e}{2}$ and get $$G_1\left(X, XY\right)=0.$$ Note that by fixing $b$ and $e$ and increasing both $c$ and $d$ by the same amount, we can change the value of $X$ while fixing $Y$. Therefore, each homogeneous part of $G_1$ must also be zero at all such $(X,XY)$. Let $G_2$ be any homogeneous part of $G_1$, of some degree $h$, and divide $G_2$ by the $h$th power of the first variable. Then we get a univariate polynomial which vanishes at all $Y$ whenever $X\neq 0$. For any $z\in \mathbb{Z}$, there is a choice of $b,c,d,e$ such that $Y=z/2$ and $X\neq 0$. Therefore, this univariate polynomial must be zero. Consequently, $G_2=0$ and $G_1=0$, which contradicts $G\neq 0$, so $G=0$ and $F=0$. 
\end{proof}

It might be possible to use the same method to prove that there is no polynomial relation between values of a character at more classes. However, since there are infinitely many classes, we might need a different method to completely prove that there is no such relation for any number of classes.

\begin{conj}
There is no polynomial relation between $n$ and the values of $\chi/\chi(1)$ at finitely many classes that is satisfied by all large integers $n$ and all $\chi\in \Irr(S_n)$, other than the polynomials in the ideal generated by $\{t_\lambda - T_\lambda\mid \lambda\text{ any partition}\}$ in the polynomial ring $\mathbb{Q}(N)[t_{(2)},t_{(3)},t_{(2^2)},t_{(4)},\dots]$.
\end{conj}

\subsection*{Acknowledgment}
I would like to thank Hung P. Tong-Viet for introducing \autoref{conjectures} to me. I also thank the anonymous person mentioned in \autoref{Remark2}.
\newpage
{\scriptsize
\begin{longtable}[c]{| c | c | c | c |}\caption{Forbidden sets of sizes $\leq 4$ of zeros that can be obtained as compositions of $4$ or less transpositions.\label{forbiddentable}}\\
\hline
\# &  Classes & A polynomial in the Gr\"obner basis & Condition on $n$\\
\hline
1 & $(2),(3),(2^2)$  &  $1$ & All $n$ \\
\hline
2 & $(3),(2^2)$ &  $t_{(2)}^2 - \frac{2}{n(n-1)}$ & $\frac{n(n-1)}{2}$ is not a square\\
\hline
3 & $(2), (2^2)$ & $t_{(3)} + \frac{1}{2(n-2)}$ & $n\equiv 2$ mod $3$\\
\hline
4 & $(2), (3)$ & $t_{(2^2)}-\frac{2}{(n-2)(n-3)}$ & $n\equiv 2$ or $3$ mod $4$\\
\hline
5 & \begin{tabular}{c}$(5),(3^2),$\\one of \\$(2),(2^2),(3)$\end{tabular} & \scriptsize\begin{tabular}{c}$t_{(2^2)} + \frac{4}{n - 3}t_{(3)} + \frac{-n(n-1)}{(n-2)(n-3)}t_{(2)}^2 + \frac{2}{(n-2)(n-3)}$, \\$t_{(3)}^2 + \frac{-9n + 60}{n(n-1)(n-2)}t_{(3)} - \frac{9}{(n-2)^2}t_{(2)}^2 + \frac{-3(n-8)}{n(n-1)(n-2)^2}$\end{tabular} & All large $n$\\
\hline
6 & \begin{tabular}{c}$(2),(5),(4,2)$,\\one of \\$(4),(3,2),(2^3)$\end{tabular} & $t_{(3)}-\frac{1}{(n-2)(n-5)}$ & $n\notin \{7,15,25\}$\\
\hline
7 & $(2^3),(5),(2^4),(3^2)$ & \scriptsize\begin{tabular}{c}$t_{(2)}^8 + \frac{-8 n^2 - 64 n + 160}{n^2(n-1)^2}t_{(2)}^6 + \frac{8 n^4 + 96 n^3 - 632 n^2 + 11040 n - 29696}{n^4(n-1)^4}t_{(2)}^4$\\$ + \frac{(32 n^6 - 640 n^5 - 288 n^4 + 102912 n^3 - 1086848 n^2 + 4042240 n - 4915200)}{n^6(n-1)^6}t_{(2)}^2 $ \\ $+ \frac{16 n^6 - 880 n^5 + 22256 n^4 - 307472 n^3 + 2187136 n^2 - 7194880 n + 8601600}{n^7(n-1)^7}$\end{tabular} & All large $n$\\
\hline
8 & $(2^2),(4,2),(4),(3^2)$ & \scriptsize$t_{(2)}^4 + \frac{-4n^2 + 4 n - 12}{n^2(n-1)^2}t_{(2)}^2 + \frac{4 n^2 - 52 n + 120}{n^3(n-1)^3}$ & All large $n$\\
\hline
9 & $(3,2^2),(5),(4),(3^2)$ & \scriptsize\begin{tabular}{c}$t_{(2)}^6 + \frac{(1/3 n^3 - 11 n^2 + 584/3 n - 720)}{n^2(n-1)^2}t_{(2)}^4 $ \\ $+ \frac{(-4/3 n^5 + 208/3 n^4 - 176/3 n^3 - 14476/3 n^2 + 20320 n - 14400)}{n^4(n-1)^4}t_{(2)}^2 $\\$ + \frac{(4/3 n^5 - 352/3 n^4 + 8252/3 n^3 - 74240/3 n^2 + 90944 n - 115200)}{n^5(n-1)^5}$\end{tabular} & All large $n$\\
\hline
10 & \scriptsize$(2^4),(4,2),(3,2),(3^2)$ & \tiny\begin{tabular}{c}$t_{(2)}^8 + \frac{(-52/3 n^2 + 1100/9 n - 4208/9)}{n^2(n-1)^2}t_{(2)}^6 $\\$+ \frac{(292/3 n^4 - 43912/27 n^3 + 94364/9 n^2 - 734288/27 n + 340736/9)}{n^4(n-1)^4}t_{(2)}^4$ \\ $ + \frac{(-320 n^6 + 46400/3 n^5 - 887488/3 n^4 + 3030592 n^3 - 50492288/3 n^2 + 139409920/3 n - 49561600)}{n^6(n-1)^6}t_{(2)}^2 $ \\ $+ \frac{(256 n^6 - 23296 n^5 + 628480 n^4 - 7422208 n^3 + 42434560 n^2 - 115072000 n + 118272000)}{n^7(n-1)^7}$\end{tabular} & All large $n$\\
\hline
11 & $(5),(2^4),(2^2),(4,2)$ & \scriptsize$t_{(2)}^4 + \frac{(-108 n + 288)}{n^2(n-1)^2}t_{(2)}^2 + \frac{(152 n - 456)}{n^3(n-1)^3}$ & All large $n$\\
\hline
12 & $(2^3),(3),(2^4),(4,2)$ & \scriptsize$t_{(2)}^4 + \frac{(2/3 n^2 - 150 n + 2440/3)}{n^2(n-1)^2}t_{(2)}^2 + \frac{(-40/3 n^2 + 440 n - 5600/3)}{n^3(n-1)^3}$ & All large $n$\\
\hline
13 & $(2),(5),(2^4),(4)$ & \scriptsize$t_{(3)}^2 + \frac{(n^2 - 31 n + 140)}{n(n-1)(n-2)}t_{(3)} + \frac{(1/4 n^2 - 33/4 n + 49)}{n(n-1)(n-2)^2}$ & All large $n$\\
\hline
14 & $(3),(5),(2^4),(4)$ & \scriptsize$t_{(2)}^4 + \frac{(-12 n^2 + 268 n - 1296)}{n^2(n-1)^2}t_{(2)}^2 + \frac{(12 n^2 - 396 n + 2352)}{n^3(n-1)^3}$ & All large $n$\\
\hline
15 & $(2),(5),(2^4),(3,2)$ & \scriptsize$t_{(3)}^2 + \frac{(n^2 - 31 n + 140)}{n(n-1)(n-2)}t_{(3)} + \frac{(1/4 n^2 - 33/4 n + 49)}{n(n-1)(n-2)^2}$ & All large $n$\\
\hline
16 & $(3,2^2),(5),(2^4),(3,2)$ & \tiny\begin{tabular}{c}$t_{(2)}^8 + \frac{(-16 n^2 + 284/5 n - 160)}{n^2(n-1)^2}t_{(2)}^6 + \frac{(60 n^4 + 464/5 n^3 - 36724/5 n^2 + 221528/5 n -370464/5)}{n^4(n-1)^4}t_{(2)}^4$\\$ + \frac{(-48 n^6 - 7872/5 n^5 + 169008/5 n^4 - 816288/5 n^3 - 789216/5 n^2 + 12158208/5 n - 3663360)}{n^6(n-1)^6}t_{(2)}^2 $ \\ $+ \frac{(6048/5 n^5 - 120096/5 n^4 + 602208/5 n^3 + 984096/5 n^2 - 12877056/5 n + 4354560)}{n^7(n-1)^7}$\end{tabular} & All large $n$\\
\hline
17 & $(3,2^2),(5),(4,2),(3^2)$ & \scriptsize\begin{tabular}{c}$t_{(2)}^6 + \frac{(1/3 n^3 - 11 n^2 - 232/3 n + 304)}{n^2(n-1)^2}t_{(2)}^4 $ \\ $+ \frac{(-4/3 n^5 + 16/3 n^4 + 784/3 n^3 + 2996/3 n^2 - 13600 n + 24000)}{n^4(n-1)^4}t_{(2)}^2$\\$ + \frac{(4/3 n^5 + 32/3 n^4 - 772/3 n^3 - 6080/3 n^2 + 20800 n - 38400)}{n^5(n-1)^5}$ \end{tabular}& All large $n$\\
\hline
18 & $(3,2^2),(2^2),(4),(3^2)$ & \scriptsize$t_{(2)}^4 + \frac{(-4 n^2 + 4 n - 12)}{n^2(n-1)^2}t_{(2)}^2 + \frac{(4 n^2 - 52 n + 120)}{n^3(n-1)^3}$ & All large $n$\\
\hline
19 & $(5),(4,2),(3,2),(3^2)$ & \scriptsize\begin{tabular}{c}$t_{(2)}^6 + \frac{(1/3 n^3 - n^2 - 412/3 n + 384)}{n^2(n-1)^2}t_{(2)}^4 $ \\ $+ \frac{(-32 n^4 + 176 n^3 + 3824 n^2 - 25088 n + 38400)}{n^4(n-1)^4}t_{(2)}^2 + \frac{(768 n^3 - 11776 n^2 + 54016 n- 76800)}{n^5(n-1)^5}$ \end{tabular}& All large $n$\\
\hline
20 & $(2^3),(5),(2^2),(4,2)$ & \scriptsize$t_{(2)}^4 + \frac{(-72 n + 180)}{n^2(n-1)^2}t_{(2)}^2 + \frac{(80 n - 240)}{n^3(n-1)^3}$ & All large $n$\\
\hline
21 & $(3,2^2),(3),(5),(3,2)$ & \scriptsize$t_{(2)}^2 + \frac{1/3 n - 8/3}{n(n-1)}$ & All large $n$\\
\hline
22 & $(2^3),(3,2^2),(3),(5)$ & \scriptsize$t_{(2)}^4 + \frac{-6 n^2 + 30 n + 200}{n^2(n-1)^2}t_{(2)}^2 + \frac{80 n - 640}{n^3(n-1)^3}$ & All large $n$\\
\hline
23 & $(3),(5),(4,2),(3,2)$ & \scriptsize$t_{(2)}^2 - \frac{8}{n(n-1)(n+2)}$ & All large $n$\\
\hline
24 & $(3,2^2),(3),(2^4),(4,2)$ & \scriptsize$t_{(2)}^4 + \frac{(-12 n^2 - 36 n + 560)}{n^2(n-1)^2}t_{(2)}^2 + \frac{12 n^2 + 212 n - 1360}{n^3(n-1)^3}$ & All large $n$\\
\hline
25 & $(3),(5),(2^4),(3,2)$ & \scriptsize$t_{(2)}^4 + \frac{-12 n^2 + 108 n - 976}{n^2(n-1)^2}t_{(2)}^2 + \frac{12 n^2 - 396 n + 2352}{n^3(n-1)^3}$ & All large $n$\\
\hline
26 & $(2^3),(3,2^2),(5),(2^4)$ & \scriptsize\begin{tabular}{c}$t_{(2)}^8 + \frac{20 n^2 - 416 n + 1424}{n^2(n-1)^2}t_{(2)}^6 $\\$+ \frac{-100 n^4 + 816 n^3 + 14268 n^2 - 142952 n + 280800}{n^4(n-1)^4}t_{(2)}^4 $\\$+ \frac{6160 n^5 - 119360 n^4 + 598160 n^3 + 590240 n^2 - 9331200 n + 14400000}{n^6(n-1)^6}t_{(2)}^2 $\\$+ \frac{-5600 n^5 + 111200 n^4 - 557600 n^3 - 911200 n^2 + 11923200 n - 20160000}{n^7(n-1)^7}$\end{tabular} & All large $n$\\
\hline
27 & $(2),(3,2^2),(5),(4)$ & \scriptsize$t_{(3)}^2 + \frac{1/2 n^2 - 61/2 n + 150}{n(n-1)(n-2)}t_{(3)} + \frac{-6 n + 48}{n(n-1)(n-2)^2}$ & All large $n$\\
\hline
28 & $(3,2^2),(2^2),(4,2),(3^2)$ & \scriptsize$t_{(2)}^4 + \frac{-4 n^2 + 4 n - 12}{n^2(n-1)^2}t_{(2)}^2 + \frac{4 n^2 - 52 n + 120}{n^3(n-1)^3}$ & All large $n$\\
\hline
29 & $(2^3),(5),(2^4),(2^2)$ & \scriptsize$t_{(2)}^4 + \frac{-60 n + 144}{n^2(n-1)^2}t_{(2)}^2 + \frac{56 n - 168}{n^3(n-1)^3}$ & All large $n$\\
\hline
30 & $(2^3),(3),(4,2),(3^2)$ & \scriptsize$t_{(2)}^4 + \frac{-6 n^2 + 70 n - 280}{n^2(n-1)^2}t_{(2)}^2 + \frac{-320/3 n + 1600/3}{n^3(n-1)^3}$ & All large $n$\\
\hline
31 & $(2),(2^3),(2^4),(3^2)$ & \scriptsize$t_{(3)}^2 + \frac{-9/5 n^2 + 153/5 n - 84}{n(n-1)(n-2)}t_{(3)} + \frac{-9/20 n^2 + 129/20 n - 21}{n(n-1)(n-2)^2}$ & All large $n$\\
\hline
32 & $(2),(2^4),(4,2),(3^2)$ & $1$ & All $n$\\
\hline
33 & $(3),(2^4),(4,2),(3^2)$ & \scriptsize$t_{(2)}^4 + \frac{(-12 n^2 + 268 n - 1264)}{n^2(n-1)^2}t_{(2)}^2 + \frac{12 n^2 - 1796/3 n + 8080/3}{n^3(n-1)^3}$ & All large $n$\\
\hline
34 & $(3,2^2),(3),(2^4),(3,2)$ & \scriptsize$t_{(2)}^4 + \frac{-12 n^2 + 108 n - 304}{n^2(n-1)^2}t_{(2)}^2 + \frac{12 n^2 - 172 n + 560}{n^3(n-1)^3}$ & All large $n$\\
\hline
35 & $(3,2^2),(5),(4,2),(3,2)$ & \scriptsize\begin{tabular}{c}$t_{(2)}^6 + \frac{-2 n^3 - 114 n^2 + 1724 n - 3984}{n^2(n-1)^2(n-10)}t_{(2)}^4 $\\$+ \frac{272 n^4 + 1760 n^3 - 61168 n^2 + 296448 n - 403200}{n^4(n-1)^4(n-10)}t_{(2)}^2 $\\$+ \frac{-8448 n^3 + 116736 n^2 - 504576 n + 691200}{n^5(n-1)^5(n-10)}$\end{tabular} & All large $n$\\
\hline
36 & $(5),(2^4),(4,2),(3,2)$ & \scriptsize\begin{tabular}{c}$t_{(2)}^8 + \frac{-12 n^2 + 28 n - 464}{n^2(n-1)^2}t_{(2)}^6 $\\$+ \frac{12 n^4 + 1032 n^3 - 18596 n^2 + 151952 n - 315648)}{n^4(n-1)^4}t_{(2)}^4 $\\$+ \frac{-1152 n^5 + 1152 n^4 + 587904 n^3 - 7424640 n^2 + 29508096 n - 36864000}{n^6(n-1)^6}t_{(2)}^2 $\\$+ \frac{27648 n^4 - 1078272 n^3 + 11363328 n^2 - 44485632 n + 58060800}{n^7(n-1)^7}$\end{tabular} & All large $n$\\
\hline
37 & $(2^4),(2^2),(3,2),(3^2)$ & \scriptsize$t_{(2)}^2 + \frac{-7/2 n + 35}{n(n-1)(n-16)}$ & All large $n$\\
\hline
38 & $(3,2^2),(5),(2^4),(4,2)$ & \scriptsize\begin{tabular}{c}$t_{(2)}^8 + \frac{-12 n^2 - 128 n + 784}{n^2(n-1)^2}t_{(2)}^6 $\\$+ \frac{12 n^4 + 1616 n^3 - 6580 n^2 - 45448 n + 144480}{n^4(n-1)^4}t_{(2)}^4 $\\$+ \frac{-3120 n^5 + 5696 n^4 + 108240 n^3 + 518944 n^2 - 5452800 n + 8870400}{n^6(n-1)^6}t_{(2)}^2 $\\$+ \frac{1824 n^5 + 3936 n^4 - 98080 n^3 - 920672 n^2 + 7960320 n - 13708800}{n^7(n-1)^7}$ \end{tabular}& All large $n$\\
\hline
39 & $(5),(2^2),(4,2),(3,2)$ & \scriptsize$t_{(2)}^4 + \frac{-2 n^2 - 70 n + 192}{n^2(n-1)^2}t_{(2)}^2 + \frac{96 n - 288}{n^3(n-1)^3}$ & All large $n$\\
\hline
40 & $(2^3),(3,2^2),(2^4),(3^2)$ & \scriptsize\begin{tabular}{c}$t_{(2)}^6 + \frac{67/7 n^3 - 4457/21 n^2 + 30554/21 n - 18696/7}{(n^5 - 58/7 n^4 + 95/7 n^3 - 44/7 n^2}t_{(2)}^4 $\\$+ \frac{144/7 n^5 - 17008/21 n^4 + 256856/21 n^3 - 1787896/21 n^2 + 1827520/7 n - 1987200/7}{n^9 - 72/7 n^8 + 218/7 n^7 - 292/7 n^6 + 183/7 n^5 - 44/7 n^4}t_{(2)}^2$\\$ + \frac{-220/7 n^5 + 3640/3 n^4 - 372140/21 n^3 + 361960/3 n^2 - 2636800/7 n + 432000}{n^11 - 79/7 n^10 + 290/7 n^9 - 510/7 n^8 + 475/7 n^7 - 227/7 n^6 + 44/7 n^5}$\end{tabular} & All large $n$\\
\hline
41 & $(2),(3,2^2),(5),(3,2)$ & \scriptsize$t_{(3)}^2 + \frac{1/2 n^2 - 61/2 n + 150}{n(n-1)(n-2)}t_{(3)} + \frac{-6 n + 48}{n(n-1)(n-2)^2}$ & All large $n$\\
\hline
42 & $(3),(2^4),(4,2),(4)$ & \scriptsize$t_{(2)}^4 + \frac{-12 n^2 + 268 n - 960}{n^2(n-1)^2}t_{(2)}^2 + \frac{12 n^2 - 396 n + 1680}{n^3(n-1)^3}$ & All large $n$\\
\hline
43 & $(2),(2^3),(3,2^2),(5)$ & \scriptsize$t_{(3)}^2 + \frac{1/2 n^2 - 61/2 n + 150}{n(n-1)(n-2)}t_{(3)} + \frac{-6 n + 48}{n(n-1)(n-2)^2}$ & All large $n$\\
\hline
44 & $(2^3),(3),(5),(4,2)$ & \scriptsize$t_{(2)}^4 + \frac{-6 n^2 + 70 n + 40}{n^2(n-1)^2}t_{(2)}^2 - \frac{320}{n^3(n-1)^3}$ & All large $n$\\
\hline
45 & $(2),(4,2),(4),(3^2)$ & \scriptsize$t_{(3)}^2 + \frac{15}{n(n-1)(n-2)}t_{(3)} + \frac{-3 n + 15}{n(n-1)(n-2)^2}$ & All large $n$\\
\hline
46 & $(2^3),(3,2^2),(5),(4,2)$ & \scriptsize\begin{tabular}{c}$t_{(2)}^8 + \frac{-6 n^2 - 182 n + 904}{n^2(n-1)^2}t_{(2)}^6$\\$ + \frac{728 n^3 + 5240 n^2 - 93232 n + 206400}{n^4(n-1)^4}t_{(2)}^4 $\\$+ \frac{-480 n^5 - 9760 n^4 - 57120 n^3 + 2252320 n^2 - 10809600 n + 14400000}{n^6(n-1)^6}t_{(2)}^2 $\\$+ \frac{281600 n^3 - 3891200 n^2 + 16819200 n - 23040000}{n^7(n-1)^7}$\end{tabular} & All large $n$\\
\hline
47 & $(5),(2^4),(4,2),(3^2)$ & \scriptsize\begin{tabular}{c}$t_{(2)}^8 + \frac{-24 n^2 + 320 n - 4192}{n^2(n-1)^2}t_{(2)}^6$\\$ + \frac{168 n^4 - 5664 n^3 + 37800 n^2 + 360864 n - 1224704}{n^4(n-1)^4}t_{(2)}^4$\\$ + \frac{-288 n^6 + 11904 n^5 - 1248 n^4 - 896512 n^3 - 5894016 n^2 + 56788480 n - 93388800}{n^6(n-1)^6}t_{(2)}^2$\\$ + \frac{144 n^6 - 7920 n^5 - 60816 n^4 + 826992 n^3 + 9817984 n^2 - 82981120 n + 144998400}{n^7(n-1)^7}$\end{tabular} & All large $n$\\
\hline
48 & $(2^3),(5),(2^4),(4,2)$ & \scriptsize\begin{tabular}{c}$t_{(2)}^8 + \frac{-236 n + 1024}{n^2(n-1)^2}t_{(2)}^6 $\\$+ \frac{136 n^3 + 12448 n^2 - 121992 n + 244000}{n^4(n-1)^4}t_{(2)}^4 $\\$+ \frac{-4480 n^4 - 218880 n^3 + 3352960 n^2 - 13804800 n + 17280000}{n^6(n-1)^6}t_{(2)}^2$\\$ + \frac{-12800 n^4 + 499200 n^3 - 5260800 n^2 + 20595200 n - 26880000}{n^7(n-1)^7}$\end{tabular} & All large $n$\\
\hline
49 & $(3,2^2),(4,2),(3,2),(3^2)$ & \scriptsize\begin{tabular}{c}$t_{(2)}^6 + \frac{22/3 n^3 - 222 n^2 + 4388/3 n - 2448}{n^2(n-1)^2(n-6)}t_{(2)}^4 $\\$+ \frac{-32/3 n^5 - 1360/3 n^4 + 46304/3 n^3 - 400112/3 n^2 + 446720 n - 518400}{n^4(n-1)^4(n-6)}t_{(2)}^2 $\\$+ \frac{-128/3 n^5 + 8576/3 n^4 - 148096/3 n^3 + 1029760/3 n^2 - 1043200 n + 1152000}{n^5(n-1)^5(n-6)}$\end{tabular} & All large $n$\\
\hline
50 & $(3,2^2),(2^4),(2^2),(3^2)$ & \scriptsize$t_{(2)}^2 + \frac{-n + 10}{n(n-1)(n-6)}$ & All large $n$\\
\hline
51 & $(2^3),(3,2^2),(3),(4,2)$ & \scriptsize$t_{(2)}^4 + \frac{-6 n^2 - 90 n + 680}{n^2(n-1)^2}t_{(2)}^2 + \frac{320 n - 1600}{n^3(n-1)^3}$ & All large $n$\\
\hline
52 & $(3,2^2),(5),(2^4),(4)$ & \scriptsize\begin{tabular}{c}$t_{(2)}^8 + \frac{-12 n^2 + 32 n + 816}{n^2(n-1)^2}t_{(2)}^6 $\\$+ \frac{12 n^4 - 624 n^3 + 15052 n^2 - 114504 n + 212832}{n^4(n-1)^4}t_{(2)}^4 $\\$+ \frac{2640 n^5 - 60480 n^4 + 358992 n^3 + 108960 n^2 - 5018112 n + 7948800}{n^6(n-1)^6}t_{(2)}^2 $\\$+ \frac{-2016 n^5 + 40032 n^4 - 200736 n^3 - 328032 n^2 + 4292352 n - 7257600}{n^7(n-1)^7}$ \end{tabular}& All large $n$\\
\hline
53 & $(2^3),(3),(2^4),(3^2)$ & \scriptsize$t_{(2)}^4 + \frac{-4 n^2 + 4 n + 48}{n^2(n-1)^2}t_{(2)}^2 + \frac{-4 n^2 + 172/3 n - 560/3}{n^3(n-1)^3}$ & All large $n$\\
\hline
54 & $(3,2^2),(5),(3,2),(3^2)$ & \scriptsize\begin{tabular}{c}$t_{(2)}^6 + \frac{1/3 n^3 + n^2 - 400/3 n + 336}{n^2(n-1)^2}t_{(2)}^4 $\\$+ \frac{4/3 n^5 - 160/3 n^4 + 632/3 n^3 + 12964/3 n^2 - 27808 n + 43200}{n^4(n-1)^4}t_{(2)}^2 $\\$+ \frac{4/3 n^5 - 352/3 n^4 + 8252/3 n^3 - 74240/3 n^2 + 90944 n - 115200}{n^5(n-1)^5}$\end{tabular} & All large $n$\\
\hline
55 & $(2^4),(2^2),(4,2),(3^2)$ & \scriptsize$t_{(2)}^4 + \frac{-38/5 n^2 + 182/5 n - 384/5}{n^2(n-1)^2}t_{(2)}^2 + \frac{38/5 n^2 - 494/5 n + 228}{n^3(n-1)^3}$ & All large $n$\\
\hline
56 & $(2),(3,2^2),(5),(2^4)$ & $1$ & All $n$\\
\hline
57 & $(2^4),(2^2),(4),(3^2)$ & \scriptsize$t_{(2)}^4 + \frac{-14/5 n - 48/5}{n^2(n-1)}t_{(2)}^2 + \frac{14/5 n^2 - 182/5 n + 84}{n^3(n-1)^3}$ & All large $n$\\
\hline
58 & $(2),(4,2),(3,2),(3^2)$ & \scriptsize$t_{(3)}^2 + \frac{15}{n(n-1)(n-2)}t_{(3)} + (-3 n + 15)/n(n-1)(n-2)^2$ & All large $n$\\
\hline
59 & $(2),(5),(2^4),(4,2)$ & $1$ & All $n$\\
\hline
60 & $(3,2^2),(5),(2^4),(2^2)$ & \scriptsize$t_{(2)}^4 + \frac{-68 n + 168}{n^2(n-1)^2}t_{(2)}^2 + \frac{72 n - 216}{n^3(n-1)^3}$ & All large $n$\\
\hline
61 & $(2^3),(5),(4,2),(3^2)$ & \scriptsize\begin{tabular}{c}$t_{(2)}^8 + \frac{-12 n^2 + 32 n - 496}{n^2(n-1)^2}t_{(2)}^6 $\\$+ \frac{36 n^4 - 624 n^3 - 1332 n^2 + 84384 n - 219200}{n^4(n-1)^4}t_{(2)}^4 $\\$+ \frac{9600 n^4 + 66560 n^3 - 2759040 n^2 + 13926400 n - 19200000}{n^6(n-1)^6}t_{(2)}^2 $\\$+ \frac{-307200 n^3 + 4710400 n^2 - 21606400 n + 30720000}{n^7(n-1)^7}$\end{tabular} & All large $n$\\
\hline
62 & $(5),(2^4),(4),(3^2)$ & \scriptsize\begin{tabular}{c}$t_{(2)}^8 + \frac{-24 n^2 + 320 n - 5472}{n^2(n-1)^2}t_{(2)}^6$\\$ + \frac{168 n^4 - 5664 n^3 + 44200 n^2 - 265056 n + 930816}{n^4(n-1)^4}t_{(2)}^4$\\$+ \frac{-288 n^6 + 11904 n^5 - 93408 n^4 - 1408512 n^3 + 19357824 n^2 - 68590080 n + 66355200}{n^6(n-1)^6}t_{(2)}^2 $\\$+ \frac{144 n^6 - 7920 n^5 + 200304 n^4 - 2767248 n^3 + 19684224 n^2 - 64753920 n + 77414400}{n^7(n-1)^7}$\end{tabular} & All large $n$\\
\hline
63 & $(3,2^2),(2^4),(4),(3^2)$ & \scriptsize\begin{tabular}{c}$t_{(2)}^8 + \frac{-35/2 n^2 + 235/2 n - 378}{n^2(n-1)^2}t_{(2)}^6 $\\$+ \frac{167/2 n^4 - 1885 n^3 + 31321/2 n^2 - 52931 n + 68868}{n^4(n-1)^4}t_{(2)}^4 $\\$+ \frac{-132 n^6 + 4484 n^5 - 57000 n^4 + 346260 n^3 - 1086860 n^2 + 1808160 n - 1425600}{n^6(n-1)^6}t_{(2)}^2 $\\$+ \frac{66 n^6 - 2878 n^5 + 49954 n^4 - 439442 n^3 + 2057900 n^2 - 4862400 n + 4536000}{n^7(n-1)^7}$\end{tabular} & All large $n$\\
\hline
64 & $(3,2^2),(3),(5),(2^4)$ & \scriptsize$t_{(2)}^4 + \frac{-12 n^2 + 108 n - 16}{n^2(n-1)^2}t_{(2)}^2 + \frac{12 n^2 - 76 n - 208}{n^3(n-1)^3}$ & All large $n$\\
\hline
65 & $(5),(2^4),(3,2),(3^2)$ & \scriptsize\begin{tabular}{c}$t_{(2)}^8 + \frac{-24 n^2 + 240 n - 1312}{n^2(n-1)^2}t_{(2)}^6$\\$ + \frac{168 n^4 - 8992/3 n^3 + 17640 n^2 - 5408/3 n - 31744}{n^4(n-1)^4}t_{(2)}^4 $\\$+ \frac{-288 n^6 + 3264 n^5 + 85152 n^4 - 1348032 n^3 + 5311104 n^2 - 4154880 n - 7372800}{n^6(n-1)^6}t_{(2)}^2 $\\$+ \frac{144 n^6 - 7920 n^5 + 200304 n^4 - 2767248 n^3 + 19684224 n^2 - 64753920 n + 77414400}{n^7(n-1)^7}$ \end{tabular}& All large $n$\\
\hline
66 & $(2^3),(3,2^2),(5),(3^2)$ & \scriptsize\begin{tabular}{c}$t_{(2)}^8 + \frac{-12 n^2 - 18 n + 4}{n^2(n-1)^2}t_{(2)}^6 $\\$+ \frac{36 n^4 - 292/3 n^3 - 2272 n^2 + 89152/3 n - 67200}{n^4(n-1)^4}t_{(2)}^4 $\\$+ \frac{-1400/3 n^5 + 4640/3 n^4 + 352520/3 n^3 - 4460720/3 n^2 + 5830400 n - 7200000}{n^6(n-1)^6}t_{(2)}^2 $\\$+ \frac{-400/3 n^5 + 35200/3 n^4 - 825200/3 n^3 + 7424000/3 n^2 - 9094400 n + 11520000}{n^7(n-1)^7}$ \end{tabular}& All large $n$\\
\hline
67 & $(2^3),(2^4),(4,2),(3^2)$ & \scriptsize\begin{tabular}{c}$t_{(2)}^6 + \frac{-7/33 n^3 - 3521/33 n^2 + 4082/3 n - 99784/33}{n^5 - 114/11 n^4 + 195/11 n^3 - 92/11 n^2}t_{(2)}^4 + $\\$\frac{16/33 n^5 - 3104/33 n^4 + 48496/11 n^3 - 162368/3 n^2 + 654400/3 n - 275200}{(n^9 - 136/11 n^8 + 434/11 n^7 - 596/11 n^6 + 379/11 n^5 - 92/11 n^4)}t_{(2)}^2 $\\$+ \frac{-160/33 n^5 + 13760/33 n^4 - 108000/11 n^3 + 3018880/33 n^2 - 11427200/33 n + 448000}{(n^11 - 147/11 n^10 + 570/11 n^9 - 1030/11 n^8 + 975/11 n^7 - 471/11 n^6 + 92/11 n^5)}$\end{tabular} & All large $n$\\
\hline
68 & $(3,2^2),(4,2),(4),(3^2)$ & \scriptsize\begin{tabular}{c}$t_{(2)}^6 + \frac{2/3 n^3 - 40 n^2 + 862/3 n - 612}{n^2(n-1)^2}t_{(2)}^4$\\$ + \frac{-8/3 n^5 + 452/3 n^4 - 4024/3 n^3 + 12412/3 n^2 - 5920 n + 7200}{n^4(n-1)^4}t_{(2)}^2 $\\$+ \frac{8/3 n^5 - 536/3 n^4 + 9256/3 n^3 - 64360/3 n^2 + 65200 n - 72000}{n^5(n-1)^5}$\end{tabular} & All large $n$\\
\hline
69 & $(3,2^2),(5),(2^2),(3,2)$ & \scriptsize$t_{(2)}^4 + \frac{-2 n^2 - 58 n + 156}{n^2(n-1)^2}t_{(2)}^2 + \frac{72 n - 216}{n^3(n-1)^3}$ & All large $n$\\
\hline
70 & $(2),(3,2^2),(4,2),(3^2)$ & $1$ & All $n$\\
\hline
71 & $(3,2^2),(3),(2^4),(3^2)$ & \scriptsize$t_{(2)}^4 + \frac{-12 n^2 + 108 n - 304}{n^2(n-1)^2}t_{(2)}^2 + \frac{12 n^2 - 172 n + 560}{n^3(n-1)^3}$ & All large $n$\\
\hline
72 & $(3,2^2),(2^4),(4,2),(3^2)$ & \scriptsize\begin{tabular}{c}$t_{(2)}^6 + \frac{-13 n^3 + 1267/15 n^2 + 970/3 n - 6504/5}{(n^5 - 46/5 n^4 + 77/5 n^3 - 36/5 n^2)}t_{(2)}^4 $\\$+ \frac{24 n^5 - 1888/15 n^4 - 14512/15 n^3 - 20152/15 n^2 + 192448/5 n - 70272}{(n^9 - 56/5 n^8 + 174/5 n^7 - 236/5 n^6 + 149/5 n^5 - 36/5 n^4)}t_{(2)}^2$\\$ + \frac{-12 n^5 + 712/15 n^4 + 12916/15 n^3 + 37144/15 n^2 - 48128 n + 97920}{(n^11 - 61/5 n^10 + 46 n^9 - 82 n^8 + 77 n^7 - 37 n^6 + 36/5 n^5)}$\end{tabular} & All large $n$\\
\hline
73 & $(5),(2^4),(2^2),(4)$ & \scriptsize$t_{(2)}^4 + \frac{52 n - 192}{n^2(n-1)^2}t_{(2)}^2 + \frac{-168 n + 504}{n^3(n-1)^3}$ & All large $n$\\
\hline
74 & $(2),(2^3),(4,2),(3^2)$ & \scriptsize$t_{(3)}^2 + \frac{15}{n(n-1)(n-2)}t_{(3)} + \frac{-3 n + 15}{n(n-1)(n-2)^2}$ & All large $n$\\
\hline
75 & $(2^3),(2^4),(2^2),(3^2)$ & \scriptsize$t_{(2)}^4 + \frac{14 n^2 - 158 n + 312}{n^2(n-1)^2}t_{(2)}^2 + \frac{-14 n^2 + 182 n - 420}{n^3(n-1)^3}$ & All large $n$\\
\hline
76 & $(2^3),(2^2),(4,2),(3^2)$ & \scriptsize$t_{(2)}^4 + \frac{-40 n^2 + 328 n - 660}{n^2(n-1)^2}t_{(2)}^2 + \frac{40 n^2 - 520 n + 1200}{n^3(n-1)^3}$ & All large $n$\\
\hline
77 & $(2^3),(3,2^2),(5),(2^2)$ & \scriptsize$t_{(2)}^4 + \frac{-62 n + 150}{n^2(n-1)^2}t_{(2)}^2 + \frac{60 n - 180}{n^3(n-1)^3}$ & All large $n$\\
\hline
78 & $(2),(3,2^2),(2^4),(3^2)$ & $1$ & All $n$\\
\hline
79 & $(5),(4,2),(4),(3^2)$ & \scriptsize$t_{(2)}^4 + \frac{-22 n^2 + 154 n - 300}{n^2(n-1)^2}t_{(2)}^2 + \frac{-12 n^3 + 184 n^2 - 844 n + 1200}{n^3(n-1)^3}$ & All large $n$\\
\hline
80 & $(2^4),(4,2),(4),(3^2)$ & \scriptsize\begin{tabular}{c}$t_{(2)}^8 + \frac{-24 n^2 + 536 n - 3720}{n^2(n-1)^2}t_{(2)}^6$\\$ + \frac{168 n^4 - 8688 n^3 + 107800 n^2 - 476880 n + 816336}{n^4(n-1)^4}t_{(2)}^4 $\\$+ \frac{-288 n^6 + 19680 n^5 - 344832 n^4 + 2102880 n^3 - 4595808 n^2 + 2652480 n - 1900800}{n^6(n-1)^6}t_{(2)}^2 $\\$+ \frac{144 n^6 - 13104 n^5 + 353520 n^4 - 4174992 n^3 + 23869440 n^2 - 64728000 n + 66528000}{n^7(n-1)^7}$ \end{tabular}& All large $n$\\
\hline
81 & $(3,2^2),(3),(2^4),(4)$ & \scriptsize$t_{(2)}^4 + \frac{-12 n^2 + 156 n - 400}{n^2(n-1)^2}t_{(2)}^2 + \frac{12 n^2 - 172 n + 560}{n^3(n-1)^3}$ & All large $n$\\
\hline
82 & $(3),(2^4),(4),(3^2)$ & \scriptsize$t_{(2)}^4 + \frac{-12 n^2 + 268 n - 624}{n^2(n-1)^2}t_{(2)}^2 + \frac{12 n^2 - 172 n + 560}{n^3(n-1)^3}$ & All large $n$\\
\hline
83 & $(2^3),(3,2^2),(4,2),(3^2)$ & \scriptsize\begin{tabular}{c}$t_{(2)}^6 + \frac{-46/7 n^3 - 454/21 n^2 + 21544/21 n - 18240/7}{(n^5 - 74/7 n^4 + 127/7 n^3 - 60/7 n^2)}t_{(2)}^4 $\\$+ \frac{24/7 n^5 + 1264/21 n^4 + 16696/21 n^3 - 616880/21 n^2 + 1041920/7 n - 1440000/7}{(n^9 - 88/7 n^8 + 282/7 n^7 - 388/7 n^6 + 247/7 n^5 - 60/7 n^4}t_{(2)}^2 $\\$+ \frac{1280/21 n^4 - 79360/21 n^3 + 154880/3 n^2 - 1625600/7 n + 2304000/7}{(n^11 - 95/7 n^10 + 370/7 n^9 - 670/7 n^8 + 635/7 n^7 - 307/7 n^6 + 60/7 n^5}$\end{tabular} & All large $n$\\
\hline
84 & $(2^3),(3),(5),(2^4)$ & \scriptsize$t_{(2)}^4 + \frac{-4 n^2 + 4 n + 272}{n^2(n-1)^2}t_{(2)}^2 + \frac{-4 n^2 + 132 n - 784}{n^3(n-1)^3}$ & All large $n$\\
\hline
85 & $(2^3),(3,2^2),(3),(2^4)$ & \scriptsize$t_{(2)}^4 + \frac{24 n^2 - 360 n + 1280}{n^2(n-1)^2}t_{(2)}^2 + \frac{-60 n^2 + 860 n - 2800}{n^3(n-1)^3}$ & All large $n$\\
\hline
86 & $(2^2),(4,2),(3,2),(3^2)$ & \scriptsize$t_{(2)}^4 + \frac{-10 n^2 + 82 n - 192}{n^2(n-1)^2}t_{(2)}^2 + \frac{16 n^2 - 208 n + 480}{n^3(n-1)^3}$ & All large $n$\\
\hline
87 & $(3,2^2),(3),(4,2),(3,2)$ & \scriptsize$t_{(2)}^2 + \frac{-8/3 n + 40/3}{n(n-1)(n-6)}$ & All large $n$\\
\hline
88 & $(2),(2^4),(4),(3^2)$ & \scriptsize$t_{(3)}^2 + \frac{-9/5 n^2 + 153/5 n - 84}{n(n-1)(n-2)}t_{(3)} + \frac{-9/20 n^2 + 129/20 n - 21}{n(n-1)(n-2)^2}$ & All large $n$\\
\hline
89 & $(3),(5),(2^4),(4,2)$ & \scriptsize$t_{(2)}^4 + \frac{-12 n^2 + 268 n - 656}{n^2(n-1)^2}t_{(2)}^2 + \frac{12 n^2 - 396 n + 1072}{n^3(n-1)^3}$ & All large $n$\\
\hline
90 & $(3),(2^4),(4,2),(3,2)$ & \scriptsize$t_{(2)}^4 + \frac{-12 n^2 + 192 n - 808}{n^2(n-1)^2}t_{(2)}^2 + \frac{12 n^2 - 396 n + 1680}{n^3(n-1)^3}$ & All large $n$\\
\hline
91 & $(3),(2^4),(3,2),(3^2)$ & \scriptsize$t_{(2)}^4 + \frac{-12 n^2 + 108 n - 304}{n^2(n-1)^2}t_{(2)}^2 + \frac{12 n^2 - 172 n + 560}{n^3(n-1)^3}$ & All large $n$\\
\hline
92 & $(3,2^2),(5),(2^4),(3^2)$ & \tiny\begin{tabular}{c}$t_{(2)}^8 + \frac{-24 n^2 + 120 n - 512}{n^2(n-1)^2}t_{(2)}^6 $\\$+ \frac{168 n^4 - 5392/3 n^3 + 1880 n^2 + 186592/3 n - 163584}{n^4(n-1)^4}t_{(2)}^4 $\\$+ \frac{-288 n^6 + 8992/3 n^5 + 27296/3 n^4 - 254176/3 n^3 - 3817088/3 n^2 + 8238080 n - 11980800}{n^6(n-1)^6}t_{(2)}^2 $\\$+ \frac{144 n^6 - 4240/3 n^5 - 45488/3 n^4 + 172816/3 n^3 + 5255552/3 n^2 - 11121920 n + 17510400}{n^7(n-1)^7}$\end{tabular} & All large $n$\\
\hline
93 & $(2),(2^4),(3,2),(3^2)$ & \scriptsize$t_{(3)}^2 + \frac{-9/5 n^2 + 153/5 n - 84}{n(n-1)(n-2)}t_{(3)} + \frac{-9/20 n^2 + 129/20 n - 21}{n(n-1)(n-2)^2}$ & All large $n$\\
\hline
94 & $(3),(4,2),(3,2),(3^2)$ & \scriptsize$t_{(2)}^2 + \frac{-8/3 n + 40/3}{n(n-1)(n-6)}$ & All large $n$\\
\hline
95 & $(3,2^2),(3),(4,2),(3^2)$ & \scriptsize$t_{(2)}^2 + \frac{-8/3 n + 40/3}{n(n-1)(n-6)}$ & All large $n$\\
\hline
96 & $(5),(2^4),(2^2),(3,2)$ & \scriptsize$t_{(2)}^4 + \frac{-20/7 n^2 - 376/7 n + 1056/7}{n^2(n-1)^2}t_{(2)}^2 + \frac{72 n - 216}{n^3(n-1)^3}$ & All large $n$\\
\hline
97 & $(5),(2^4),(4,2),(4)$ & \scriptsize\begin{tabular}{c}$t_{(2)}^4 + \frac{-12 n^4 + 340 n^3 - 3316 n^2 + 12508 n - 15600}{n^2(n-1)^2(n^2+2n-23)}t_{(2)}^2 $\\$+ \frac{12 n^4 - 468 n^3 + 4932 n^2 - 19308 n + 25200}{n^3(n-1)^3(n^2+2n-23)}$ \end{tabular}& All large $n$\\
\hline
98 & $(2),(3,2^2),(5),(4,2)$ & $1$ & All $n$\\
\hline
99 & $(3),(4,2),(4),(3^2)$ & \scriptsize$t_{(2)}^2 + \frac{2/3 n - 10/3}{n(n-1)}$ & All large $n$\\
\hline
100 & $(2^3),(3,2^2),(2^2),(3^2)$ & \scriptsize$t_{(2)}^4 + \frac{20 n^2 - 212 n + 420}{n^2(n-1)^2} t_{(2)}^2 + \frac{-20 n^2 + 260 n - 600}{n^3(n-1)^3}$ & All large $n$\\
\hline
101 & $(2),(2^3),(5),(2^4)$ & \scriptsize$t_{(3)}^2 + \frac{n^2 - 31 n + 140}{n(n-1)(n-2)}t_{(3)} + \frac{1/4 n^2 - 33/4 n + 49}{n(n-1)(n-2)^2}$ & All large $n$\\
\hline
\end{longtable}
}


\begin{thebibliography}{9}

\bibitem{CP} T.~Y. Chow and J. Paulhus, Algorithmically distinguishing irreducible characters of the symmetric group, Electron. J. Combin. {\bf 28} (2021), no.~2, Paper No. 2.5, 27 pp.; MR4245298

\bibitem{HSTZ} G. Heide,\ J. Saxl,\ P.~H. Tiep,\ and\ A. Zalesski, Conjugacy action, induced representations and the Steinberg square for simple groups of Lie type, Proc. Lond. Math. Soc. (3) {\bf 106} (2013), no.~4, 908--930. MR3056296

\bibitem{Isaacs} I.~M. Isaacs, {\it Character theory of finite groups}, Pure and Applied Mathematics, No. 69, Academic Press, New York, 1976. MR0460423

\bibitem{Lassalle} M. Lassalle, An explicit formula for the characters of the symmetric group, Math. Ann. {\bf 340} (2008), no.~2, 383--405; MR2368985

\bibitem{Navarro} G. Navarro, personal communications with Hung P. Tong-Viet.

\bibitem{Miller} A.~R. Miller, Character and class parameters from entries of character tables of symmetric groups, Math. Proc. Cambridge Philos. Soc. {\bf 179} (2025), no.~2, 439--447. MR4945975

\bibitem{MS} A.~R. Miller\ and\ D. Scheinerman, Large-scale Monte Carlo simulations for zeros in character tables of symmetric groups, Math. Comp. {\bf 94} (2025), no.~351, 505--515. MR4807819

\bibitem{OEIS}  OEIS Foundation Inc., The On-Line Encyclopedia of Integer Sequences, Published electronically at https://oeis.org

\bibitem{OS} J.~B. Olsson\ and\ D.~W. Stanton, Block inclusions and cores of partitions, Aequationes Math. {\bf 74} (2007), no.~1-2, 90--110. MR2355859

\bibitem{Siegel} C. L. Siegel, Über einige Anwendungen diophantischer Approximationen, Abh. Preuss. Akad. Wiss., Math. Phys. Kl. No. 1., 1929.

\end{thebibliography}
\end{document}